\documentclass[a4paper,11pt]{amsart}

\usepackage[utf8]{inputenc}

\usepackage[a4paper,right=2.75cm,left=2.75cm,bottom=2.75cm, top=2.75cm]{geometry}

\usepackage{lmodern}
\usepackage{latexsym}
\usepackage{graphics}
\usepackage{amsmath}
\usepackage{amssymb}
\usepackage{bm}
\usepackage{mathrsfs}
\usepackage{enumerate,enumitem}
\usepackage{csquotes}
\usepackage{subcaption}
\usepackage{pgfplots,tikz}
\usetikzlibrary{decorations.pathreplacing}
\pgfplotsset{compat=1.10}
\usepgfplotslibrary{fillbetween}
\usetikzlibrary{patterns}
\usetikzlibrary{intersections, calc, angles}
\usepackage{caption}  
\usepackage{enumerate}
\usepackage{cancel}
\usepackage[normalem]{ulem}
\usepackage{amsmath}
\newcommand{\st}[1]{\ifmmode\text{\sout{\ensuremath{#1}}}\else\sout{#1}\fi}

\usepackage{xfrac}
\usepackage{comment}

\usepackage{hyperref}
\usepackage[capitalise, nameinlink, noabbrev]{cleveref}

\hypersetup{colorlinks=true, 
	linkcolor=blue,
	citecolor=red,
}
\usepackage[showonlyrefs]{mathtools}

\usepackage{etoolbox,kvoptions,logreq,pdftexcmds}

\numberwithin{equation}{section} 

\newtheorem{theorem}{Theorem}[section]
\newtheorem{corollary}[theorem]{Corollary}

\newtheorem{problem}[theorem]{Problem}
\newtheorem{lemma}[theorem]{Lemma}
\newtheorem{proposition}[theorem]{Proposition}
\theoremstyle{definition} 
\newtheorem{definition}[theorem]{Definition}

\newtheorem{remark}[theorem]{Remark}

\newcommand{\R}{\mathbb{R}}	
\newcommand{\N}{\mathbb{N}} 
\newcommand{\dx}{\,\mathrm{d}x}	
\newcommand{\ds}{\,\mathrm{d}S}	
\newcommand{\dS}{\,\mathrm{d}S}	
\newcommand{\dr}{\,\mathrm{d}r}	
\renewcommand{\d}{\mathrm{d}}
\newcommand{\weak}{\rightharpoonup}
\newcommand{\nnu}{\bm{\nu}}  

\newcommand{\eps}{\varepsilon}

\newcommand{\norm}[1]{\left\lVert #1 \right\lVert}
\newcommand{\abs}[1]{\left| #1 \right|}
\newcommand{\sub}{\subseteq}

\newcommand{\tu}[1]{\textup{#1}}

\DeclareMathOperator{\dist}{\mathrm{dist}}

\newenvironment{bvp}{\left\{\begin{aligned}  }{\end{aligned}\right.}

\begin{document}
	
	\title[Structure of the free interfaces near triple junctions]{Structure of the free interfaces near triple junction singularities in harmonic maps and optimal partition problems}

	\author[R. Ognibene]{Roberto Ognibene}
	\address{Roberto Ognibene
		\newline \indent Dipartimento di Matematica e Applicazioni
		\newline \indent Universit\`a degli studi di Milano-Bicocca
		\newline\indent  Via Roberto Cozzi, 55, 20126 Milano, Italy}
	\email{roberto.ognibene@unimib.it}
	
	\author[B. Velichkov]{Bozhidar Velichkov}
	\address{Bozhidar Velichkov
		\newline \indent Dipartimento di Matematica
		\newline \indent Universit\`a di Pisa
		\newline\indent  Largo Bruno Pontecorvo, 5, 56127 Pisa, Italy}
	\email{bozhidar.velichkov@unipi.it}


	\begin{abstract}
	We consider energy-minimizing harmonic maps into trees and we prove the regularity of the singular part of the free interface near triple junction points. Precisely, by proving a new epiperimetric inequality, we show that around any point of frequency $\sfrac32$, the free interface is composed of three $C^{1,\alpha}$-smooth $(d-1)$-dimensional manifolds (composed of points of frequency $1$) with common $C^{1,\alpha}$-regular boundary (made of points of frequency $\sfrac32$) that meet along this boundary at 120 degree angles. Our results also apply to spectral optimal partition problems for the Dirichlet eigenvalues.
	\end{abstract}
	
	\maketitle

	\section{Introduction}
	This paper is dedicated to the structure of the singular set of harmonic energy-minimizing maps $u:B_1\to\Sigma$, defined on $B_1\subset\R^d$ and with values in singular spaces $\Sigma$. We consider the model case in which the target space is of the form 
	\begin{equation}\label{e:definition-of-Sigma-N}
		\Sigma_N:=\left\{X\in\R^N\colon {X_i\geq 0~\text{for any $i$ and } }X_iX_j=0~\text{for all }j\neq i\right\}\,
	\end{equation} 
     endowed with the distance
     \begin{equation*}
         \d_{\Sigma_N}(X,Y):=\begin{cases}
		 	|X_i-Y_i|,&\text{if }X_j=Y_j=0~\text{for all }j\neq i, \\
		 	|X_i|+|Y_j|,&\text{if }X_i,Y_j>0~\text{for some }i\neq j.
		 \end{cases}
    \end{equation*}
	By the works of Gromov-Schoen \cite{GS}, Sun \cite{Sun}, Caffarelli-Lin \cite{CL2007, CL2008}, Tavares-Terracini \cite{TerraciniTavares2012} and Soave-Terracini \cite{ST-opt-gap} it is known that if
	$$\mathcal F(u)=\{x\in B_1\ :\ u(x)=0\},$$
	is the nodal set of a local minimizer $u=(u_1,\dots,u_N)$ of the energy 
	\begin{equation*}	E(u,B_1):=\sum_{i=1}^N\int_{B_1}|\nabla u_i|^2\dx\,,
	\end{equation*}
	then $\mathcal F(u)$ can be decomposed into regular and singular parts according to the values of the Almgren's frequency $\gamma(u,x)$ (see \Cref{sub:intro-almgren}) as follows: 
	$$\mathrm{Reg}(u)=\big\{x\in \mathcal F(u)\ :\ \gamma(u,x)=1\big\}\qquad\text{and}\qquad \mathrm{Sing}(u)=\big\{x\in \mathcal F(u)\ :\ \gamma(u,x)\geq \sfrac{3}{2}\big\}.$$
	By \cite{GS, Sun, CL2008}, in every dimension $d\ge 2$, the regular part $\mathrm{Reg}(u)$ is a smooth $(d-1)$-dimensional manifold, while the structure of the singular set was studied in \cite{GS,Sun} and \cite{Alper,Dees}, where it was shown that $\mathrm{Sing}(u)$ has Hausdorff dimension $d-2$ (\cite{GS,Sun}) and is $(d-2)$-rectifiable (\cite{Alper,Dees}). In dimension $d=2$, Conti, Terracini and Verzini showed in \cite{CTV2003} that the singular set is discrete and gave a complete description of the nodal set $\mathcal F(u)$ around any singular point (see also the recent paper \cite{spadaro-generic} for the case of generic solutions). We also refer to the recent papers \cite{ST_partially_1,ST_partially_2} for the case of partially segregated  energy-minimizing maps with values in a singular space.

  \medskip
	
	In this paper we study the regularity of the lowest stratum of $\mathrm{Sing}(u)$, that is
	 \begin{equation}\label{e:32-singular-set}
	 	\mathcal{F}_{\sfrac32}(u):=\big\{x\in \mathcal F(u)\ :\ \gamma(u,x)=\sfrac32\big\},
	 	 	\end{equation}
and the structure of the entire free interface $\mathcal F(u)$ around singular points in $\mathcal{F}_{\sfrac32}(u)$ in any dimension $d\ge 2$. Our main result is the following:
	
	\begin{theorem}\label{t:main-main}
		Let $u:B_1\to\Sigma_N$ be an energy-minimizing map in $B_1\subset\R^d$, for some $d\ge 2$ and $N\ge 1$. Then,  the singular set of $\mathrm{Sing}{(u)}$ of the free interface $\mathcal F(u)$ can be decomposed as 
		$$\mathrm{Sing}{(u)}=\mathcal{F}_{\sfrac{3}{2}}(u)\cup\Big\{x\in \mathcal F{(u)}\ :\ \gamma(u,x)\geq \sfrac32+\delta_d\Big\},$$
        where $\delta_d>0$ is dimensional constant.	The set 
		$\mathcal{F}_{\sfrac32}(u)$
		is an open subset of $\mathrm{Sing}(u)$ and, locally, a $(d-2)$-dimensional $C^{1,\alpha}$-smooth manifold. Moreover, around any $x\in \mathcal{F}_{\sfrac32}(u)$, the nodal set $\mathcal F(u)$ consists of three $(d-1)$-manifolds $\Gamma_{12}$, $\Gamma_{23}$, $\Gamma_{31}$ with common boundary $\mathcal F_{\sfrac32}(u)$, which are all $C^{1,\alpha}$-regular up to $\mathcal F_{\sfrac32}(u)$ and form $120$ degree angles at $\mathcal F_{\sfrac32}(u)$. 
	\end{theorem}
	
	Results in the spirit of \cref{t:main-main} exist in the context of area-minimizing surfaces; in this framework, description of the triple junction singularities is known thanks to the historical works from the 70s of Jean Taylor \cite{TaylorInventiones,TaylorAnnals}, for hypersurfaces in $\R^3$, and from the early 90s of Leon Simon \cite{SimonJDG} in any dimension $d\ge 3$. The techniques from \cite{TaylorInventiones,TaylorAnnals,SimonJDG} have been abundantly explored in the last decade (see for instance \cite{modp-1,modp-2,modp-3} and the references therein).

    \medskip
    
    In the context of free boundary problems, triple junction configurations  appear on the graphs of solutions to multiple membrane problems (see for instance \cite{SavinYu2019,SavinYu2021,SavinYu2023,DeSilvaSavin2024}). The $120$-degrees-angles triple junction singularities that we investigate in this paper appear on the free interface formed in the domain of a solution and do not satisfy any a priori graphicality condition (thus our situation is more similar to the one in \cite{TaylorInventiones,TaylorAnnals,SimonJDG}); in this sense, to our knowledge, this is the first instance of a study of triple junction singularities in the contexts of free boundary problems and harmonic maps in dimension $d>2$.

    \medskip
    
	The main ingredient in the proof of \cref{t:main-main} is the following epiperimetric inequality (\cref{t:epi}) for the $\sfrac{3}{2}$-Weiss energy
    \begin{equation*}
		W_{\frac{3}{2}}(u):=\sum_{i=1}^N\int_{B_1}|\nabla u_i|^2\dx-\frac{3}{2}\sum_{i=1}^N\int_{\partial B_1}u_i^2\ds
	\end{equation*}
    near homogeneous triple junctions $Y(x)=Y(x_{d-1},x_d)$ with $2$-dimensional profile 
    defined as
		\begin{equation}\label{eq:def-Y}
			Y_i(r,\theta)=\begin{cases}
				r^{\frac{3}{2}}\abs{\cos\left(\frac{3}{2}\theta\right)},&\text{for }-\pi+\frac{2\pi}{3}(i-1)\leq \theta\leq- \pi+\frac{2\pi}{3}i, \\
				0, &\text{elsewhere},
			\end{cases}
		\end{equation}
		for $i=1,2,3$. 
    \begin{theorem}[Epiperimetric inequality]\label{t:epi}
		There exists $\delta,\eps,\tau\in (0,1)$ depending only on the dimension $d$ such that the following holds. For any $N\in\N$ and any $\sfrac{3}{2}$-homogeneous $c\in H^1(B_1;\Sigma_N)\cap C^{0,1}(B_1;\R^N)$ such that $c_{i}\geq 0$ for all $i=1,\dots,N$,
		\begin{equation*}
			\sum_{i=1}^3\d_{\mathcal{H}}\left(\{c_i>0\},\{Y_i>0\}\right)+\sum_{i=4}^N\d_{\mathcal{H}}\left(\{c_i>0\},\{x_{d-1}=x_d=0\}\right)\leq \tau,
		\end{equation*}
		where the Hausdorff distance $\d_{\mathcal{H}}$ is measured in $\overline B_1$, and
		\begin{equation*}
			\norm{\mathrm{d}_{\Sigma_N}(c,Y)}_{H^1(B_1)}\leq \delta,
		\end{equation*}
		there exists $u\in H^1(B_1;\Sigma_N)\cap C^{0,1}(B_1;\R^N)$ such that $u_i\ge 0$ for all $i$, $u=c$ on $\partial B_1$, and
		\begin{equation*}
			W_{\frac{3}{2}}(u)\leq (1-\eps)W_{\frac{3}{2}}(c).
		\end{equation*}
	\end{theorem}
    
    The proof of \Cref{t:epi} is obtained via a contradiction argument in the spirit of Weiss \cite{Weiss} (see also \cite{FSthin,GPGthin,TaylorAnnals,TaylorInventiones}).  The epiperimetric inequality approach to the analysis of the singular points in the context of harmonic maps is, to our knowledge, new in this field.\medskip

 One of the main consequences of \cref{t:epi} is the uniqueness of the ($\sfrac32$-homogeneous) blow-up at every point of frequency $\sfrac32$. We notice that the uniqueness of the tangent maps was claimed in \cite[Section 3]{CL2010}. The argument in \cite{CL2010} is based on a Monneau-type formula obtained via the Weiss-energy linearization identity from \cite[Lemma 3.2]{CL2010}. Unfortunately, there is a missing term in \cite[Lemma 3.2]{CL2010}, which compromises the conclusion in \cite{CL2010}. The correct formula is given in \cref{lemma:weiss_diff} and contains an extra term $\beta$. This term is linear and positive, and requires particular attention in the proof of the epiperimetric inequality. Finally, we also point out that the term $\beta$ from \cref{lemma:weiss_diff} is specific for our functional and contains information on the interaction between the different phases. 
\medskip

  Our epiperimetric inequality shares some common features with the epiperimetric inequalities from \cite{TaylorAnnals,TaylorInventiones}, for instance, the concept of \enquote{linearization} and the difficulties related to the \enquote{parametrization} over the limit cone.     On the other hand, our functional is of different nature, which is reflected, in particular, in the classifications of the solutions of the linearized problem in \Cref{sec:linearized}, in the Weiss-energy linearization identity from \cref{lemma:weiss_diff}, and in the constructions of the competitors in Step 1 and Step 4 of the proof of the epiperimetric inequality in  \cref{sec:proof-of-epi}, where we face structural problems caused by the segregation constraint which is not present in \cite{Weiss,GPGthin,FSthin,TaylorAnnals,TaylorInventiones}. In fact, there are linear remainder terms that appear in the Weiss-energy identity \Cref{lemma:weiss_diff} and also in the projection of the homogeneous extensions of blow-up sequences onto the space of $\sfrac{3}{2}$-homogeneous minimizers, see Step 3 of the proof of \Cref{t:epi} and, in particular, \eqref{eq:epi3}.\medskip
  
  Once we have the epiperimetric inequality \cref{t:epi}, the conclusion in \cref{t:main-main} follows mainly from the classification of the $\sfrac32$-homogeneous tangent maps (\cref{prop:32-homogeneous-classification}), from the frequency gap from above (\cref{lemma:gap_from_above}), and from a topological argument (\cref{l:no-holes}) similar to \cite{SimonJDG,TaylorAnnals,TaylorInventiones}, which guarantees the persistence of singularities over the triple junction model.

    \subsection{\texorpdfstring{Two consequences of \cref{t:main-main}}{Two consequences of the main result}}\label{sub:further-results}
We now discuss two corollaries of \cref{t:main-main} for harmonic maps with values in metric graphs and for spectral optimal partitions.

\medskip

First of all we notice that, as a corollary of \cref{t:main-main} and the clean-up result of Sun \cite{Sun}, we obtain the following result for energy-minimizing maps with values in $\R$-trees with infinitely many edges and vertices. We will give the precise definition of a $\R$-tree $T$ and an energy-minimizing map with values in $T$ in \cref{s:harmonic-maps-R-trees}; we only stress here that the sets $\Sigma_N$ from \eqref{e:definition-of-Sigma-N} is a particular example of an $\R$-tree and that, for general $\R$-trees $T$ and energy-minimizing maps $u:B_1\to T$, the Almgren's frequency $\gamma(u,x)$ from \cref{t:main-main} is replaced by the so-called order function $\textrm{Ord}^{u}(x)$ (see \cref{s:harmonic-maps-R-trees}). The key point is that, thanks to the clean-up result of Sun \cite{Sun}, even if the target tree has infinitely many vertices and infinitely many edges (see \cref{sub:R-trees-definition}) coming out of each vertex, any energy-minimizing map with values in $T$ locally \enquote{sees} only finitely many vertices and edges; at this point, thanks to the regularity results of Gromov and Schoen \cite{GS}, one can locally replace $T$ by a tree of the form $\Sigma_N$, so  \cref{t:main-main} applies locally. Our main result for energy-minimizing maps in general $\R$-trees is the following; we postpone the proof and the precise definitions to \cref{s:harmonic-maps-R-trees}. 

\begin{theorem}\label{t:harmonic-maps-graphs}
		Let $B_1$ be the unit ball in $\R^d$, for some $d\ge 2$. Let $T$ be an $\R$-tree and let $u:B_1\to T$ be an energy-minimizing map. Then, the singular set 
        $$\mathrm{Sing}{(u)}=\{x\in\Omega\ :\ {\rm Ord}^u(x)>1\},$$
        can be decomposed as
		$$\mathrm{Sing}{(u)}:=\mathcal{F}_{\sfrac32}{(u)}\cup\Big\{x\in \Omega\ :\ {\rm Ord}^u(x)\geq\sfrac32+\delta_d\Big\},$$
        where $\delta_d>0$ is a dimensional constant and where 
        $$\mathcal{F}_{\sfrac32}{(u)}:=\big\{x\in \Omega\ :\ {\rm Ord}^u(x)=\sfrac32\big\}.$$
        Moreover:
        \begin{enumerate}
        \item[{\rm (i)}] The set 
		$\mathcal{F}_{\sfrac32}{(u)}$
		is an open subset of $\mathrm{Sing}(u)$ and is locally a $(d-2)$-dimensional $C^{1,\alpha}$-smooth manifold.
        \item[\rm (ii)]
        Around any $x\in \mathcal{F}_{\sfrac32}{(u)}$, the set $u^{-1}(u(x))$ consists of three $(d-1)$-manifolds 
        with common boundary $\mathcal F_{\sfrac32}{(u)}$, which are all $C^{1,\alpha}$-regular up to $\mathcal F_{\sfrac32}{(u)}$ and form $120$ degree angles at $\mathcal F_{\sfrac32}{(u)}$. 
        \end{enumerate}
	\end{theorem}

Another consequence concerns the regularity of the following optimal partition problem studied in \cite{CL2007}, \cite{CTV2003,CTVFucick2005} and \cite{OV} (see also \cite{tavares_1,tavares_2,tavares_3} for further developments). Given a bounded open set $D\subset\R^d$, $d\ge 2$, and a natural number $N\ge 2$ we consider the optimal partition problem
\begin{equation}\label{e:optimal-partition}
\min\Big\{\sum_{j=1}^N\lambda_1(\Omega_j)\ :\ \Omega_j\,-\,\text{open},\ \Omega_j\subset D,\ \Omega_i\cap\Omega_j=\emptyset\ \text{for}\ i\neq j\Big\},
\end{equation}
where $\lambda_1(\Omega_j)$ is the first eigenvalue on $\Omega_j$ with Dirichlet boundary conditions on $\partial\Omega_j$. If the $N$-tuple of disjoint sets $(\Omega_1,\dots,\Omega_N)$ is a minimizer of \eqref{e:optimal-partition} and if $u=(u_1,\dots,u_N)$ is the vector of the normalized first eigenfunctions (each $u_j$ is extended by zero outside $\Omega_j$), then the free interface 
$$\mathcal F{(u)}:=D\cap\Big({\bigcup_{j=1}^N}\partial\Omega_j\Big)$$ 
can be decomposed as a regular part and a singular part as follows (see \cite{CL2007,ST-opt-gap,Alper}): the regular part $\mathrm{Reg}{(u)}$ consists of all points $x\in\mathcal F{(u)}$ of frequency $\gamma(u,x)=1$ and is a smooth manifold of dimension $d-1$; the singular part $\mathrm{Sing}{(u)}$ is a closed $(d-2)$-rectifiable set composed of all the points of frequency $\gamma(u,x)\geq \sfrac{3}{2}$. We have the following theorem.
	\begin{theorem}\label{t:optimal-partitions}
	Let $(\Omega_1,\dots,\Omega_N)$ be a minimizer of \eqref{e:optimal-partition} and $u=(u_1,\dots,u_N)$ be the corresponding vector of the normalized first eigenfunctions. 	Then, the singular set of $\mathrm{Sing}{(u)}$ of the free interface $\mathcal F(u)$ can be decomposed as
		$$\mathrm{Sing}{(u)}=\Big\{x\in \mathcal F{(u)}\ :\ \gamma(u,x)=\sfrac32\Big\}\cup\Big\{x\in \mathcal F{(u)}\ :\ \gamma(u,x)\geq \sfrac32+\delta_d\Big\},$$
        where $\delta_d>0$ is dimensional constant.	The set $\mathcal{F}_{\sfrac32}{(u)}:=\big\{x\in \mathcal F(u)\ :\ \gamma(u,x)=\sfrac32\big\}$
		is an open subset of $\mathrm{Sing}(u)$ and is locally a $(d-2)$-dimensional $C^{1,\alpha}$-smooth manifold.  Moreover, around any $x\in \mathcal{F}_{\sfrac32}{(u)}$, the set $\mathcal F{(u)}$ consists of three $(d-1)$-manifolds 
        with common boundary $\mathcal F_{\sfrac32}{(u)}$, which are all $C^{1,\alpha}$-regular up to $\mathcal F_{\sfrac32}{(u)}$ and form $120$ degree angles at $\mathcal F_{\sfrac32}{(u)}$. 
	\end{theorem}
    
   \subsection{The Bishop-Friedland-Hayman's  conjecture for the min-max}\label{sub:intro-bishop}

	Consider, for $p\in(0,+\infty]$, the family of optimal partition problems
		\begin{equation}\label{eq:p_partition}
		\min\left\{ \left(\sum_{i=1}^N\lambda_1(\omega_i)^p\right)^{\frac{1}{p}}\colon \omega_i\sub\partial B_1~\text{open, connected and s.t. }\omega_i\cap\omega_j=\emptyset ~\text{for }i\neq j \right\},\tag{$\text{OP}_p$}
	\end{equation}
	where $\lambda_1(\omega)$ denotes the first Dirichlet eigenvalue of the spherical Laplacian in $\omega\sub\partial B_1$, and 
	\begin{equation*}
		\left(\sum_{i=1}^N\lambda_1(\omega_i)^p\right)^{\frac{1}{p}}:=
		\max_{i\in\{1,\dots,N\}}\lambda_1(\omega_i)
        \quad\text{when}\quad p=+\infty.
	\end{equation*} 
	In the particular case of the $3$-partition, i.e. $N=3$, we have the following conjecture.
	\begin{problem}[Bishop-Friedland-Hayman-Helffer-Hoffmann-Ostenhof-Terracini]\label{conj}
		For $N=3$, prove that the $Y$-configuration is the unique minimizer of \eqref{eq:p_partition}, for any $p\in(0,+\infty]$.
	\end{problem}
	This conjecture was posed in \cite{Bishop} in the case $p=1$, but can be generalized for any $p\in(0,+\infty]$ (see \cite{FH} and  \cite{HHOT10}) and is related to shape optimization and other optimal partition problems \cite{CTV2003,CTVFucick2005,BBH1998,Helffer-survey,HH10,HHOT10}. 
    \cref{conj} was solved for $p=+\infty$ in dimension $d=3$ by Helffer, Hoffmann-Ostenhof and Terracini in \cite{HHOT10}. In \cref{sub:classification} we solve \cref{conj}, for $p=+\infty$ and in any dimension $d\ge 3$, as a consequence of the classification of $\sfrac32$-homogeneous tangent maps (first obtained in \cite{ST-opt-gap}), see \cref{cor:min-max}.

    \subsection*{Organization of the paper} The rest of the paper is organized as follows. In \Cref{sec:intro-notations} we fix the notation, the functional setting and we recall some known facts about segregated solutions. In \Cref{sec:gap}, we classify the $\sfrac{3}{2}$-homogeneous solutions and we find, as a corollary, the shape of minimizers of a certain optimal partition problem on the sphere (see \Cref{cor:min-max}). \Cref{sec:linearized} contains the classification of $\sfrac{3}{2}$-homogeneous solutions to the linearized problem over 
    $Y$-configurations. \Cref{sec:proof-of-epi} is devoted to the proof of our main tool, the epiperimetric inequality \Cref{t:epi} and, finally, in \Cref{sec:regularity} we conclude the proof of our main regularity result \Cref{t:main-main}.
    
	\section{Preliminaries}\label{sec:intro-notations}
    In this section we introduce the functional setting and we define the main tools in our framework, such as monotonicity formulas.
    
    For any $N\ge 1$, let $\Sigma_N$ be as in \eqref{e:definition-of-Sigma-N}. For any open set $D\sub \R^d$, we define the Sobolev space
	\begin{equation*}
		H^1(D;\Sigma_N)=\left\{u\in (H^1(D))^N\colon u_iu_j=0~\text{a.e. in }D~\text{for all }i\neq j\right\},
	\end{equation*}
	with the space $H^1_{\tu{loc}}(D;\Sigma_N)$ being analogously defined; we refer to \cite{KS} for the theory of Sobolev spaces in this setting. Now, to each open $D\sub\R^d$ and each function $u\in H^1(D;\Sigma_N)$ we associate the corresponding Dirichlet energy, defined as
	\begin{equation*}
		E(u,D):=\sum_{i=1}^N\int_D|\nabla u_i|^2\dx.
	\end{equation*}
	In the present paper, we work with certain non-negative stationary points of $E$, hence we first define the space of admissible functions as
	\begin{equation*}
		\mathcal{A}(D;N):=\left\{u\in H^1_{\tu{loc}}(D;\Sigma_N)\colon u_i\geq 0~\text{for all }i=1,\dots,N\right\}.
	\end{equation*}
	We consider two classes of \enquote{critical points}. The first one is the class of minimizers: we say that $u\in H^1_{\tu{loc}}(D;\Sigma_N)$ is a \emph{local minimizer}, and we write $u\in \mathcal{M}(D;N)$, if
	\begin{equation*}
		E(u,\Omega)\leq E(v,\Omega)
	\end{equation*}
	for all open $\Omega\Subset D$ and all $v\in H^1(\Omega;\Sigma_N)$ such that $u-v\in H^1_0(\Omega;\Sigma_N)$. 
    We also introduce the class $\mathcal{S}(D;N)$ of functions $u\in\mathcal{A}(D;N)$ satisfying
	\begin{equation*}
		\begin{bvp}
			-\Delta u_i &\leq 0,&&\text{in }D, \\
			-\Delta \Big(u_i-\sum_{j\neq i}u_j\Big)&\geq 0, &&\text{in }D,
		\end{bvp}
	\end{equation*}
	in the sense of distributions, for all $i=1,\dots,N$. In particular, there holds $\mathcal{M}(D;N)\sub\mathcal{S}(D;N)$ (see e.g. \cite[Theorem 5.1]{CTV2005indiana}). The converse inclusion $\mathcal{S}(D;N)\sub\mathcal{M}(D;N)$ was proved in \cite[Theorem 1.6]{wang-zhang}, so the two classes coincide: $\mathcal{S}(D;N)=\mathcal{M}(D;N)$. In what follows, we keep both notations for historical reasons, as most of the papers do distinguish between the two classes, and also in order to emphasize whether we make use of the minimality or the Euler-Lagrange equation.

    \medskip

    We now recall the main basic properties of critical points $u\in\mathcal{S}(D;N)$ and energy-minimizers $u\in\mathcal{M}(D;N)$ and of the free boundary $\mathcal{F}(u)$. Most of the properties we recall were proved in a series of papers, namely \cite{GS,CL2007} for the class $\mathcal{M}(D;N)$ and \cite{CTV2005indiana,TerraciniTavares2012,ST-opt-gap} for $\mathcal{S}(D;N)$; for what concerns the regularity of different classes of segregated solutions we refer to \cite{NRS24,ST24,ST24a}.

    \subsection{Lipschitz continuity}
     Critical points of the energy functional belonging to the class $\mathcal{S}(D;N)$ satisfy the following optimal regularity result.
    \begin{theorem}[Theorem 8.4 \cite{CTV2005indiana}]
        If $u\in \mathcal S(D;N)$, then $u$ is locally Lipschitz continuous. Moreover, if $\partial D$ is of class $C^1$ and $u_i=\phi_i$ on $\partial D$, for any $i=1,\dots,N$ and for some $\phi\in \mathcal{A}(D;N)$, then $u_i$ is locally Lipschitz continuous up to $\partial D$.
    \end{theorem}
    We point out that, in case of minimizers, the Lipschitz continuity was first proved in \cite{GS}. As a direct consequence, we have that
    \begin{equation*}
		\Omega^u_i:=\{x\in D\colon u_i(x)>0\}
	\end{equation*}
    is open and that $u_i$ is harmonic, hence analytic, in $\Omega_i^u$.

	\subsection{Almgren and Weiss monotonicity formulas}\label{sub:intro-almgren}
	For any non-trivial $u\in\mathcal{S}(D;N)$ one can define the Almgren \emph{frequency function} for any $x_0\in D$ and any $r<\dist(x_0,\partial D)$ as
	\begin{equation*}
		\mathcal{N}(u,x_0,r):=\frac{E(u,x_0,r)}{H(u,x_0,r)},
	\end{equation*}
	where
	\begin{equation*}
		E(u,x_0,r):=\frac{1}{r^{d-2}}E(u,B_r(x_0))=\frac{1}{r^{d-2}}\sum_{i=1}^N\int_{B_r(x_0)}|\nabla u_i|^2\dx
	\end{equation*}
	is the scaled energy and
	\begin{equation*}
		H(u,x_0,r):=\frac{1}{r^{d-1}}\sum_{i=1}^N\int_{B_r(x_0)}u_i^2\ds
	\end{equation*}
	is the scaled height function. From \cite[Theorem 2.2 \& Remark 2.4]{TerraciniTavares2012} we get the monotonicity of $\mathcal{N}$, i.e. we have the following.
    \begin{theorem}
        Let $u\in\mathcal{S}(D;N)$ be such that $u\not\equiv 0$ in $D$. Then for any $x_0\in D$, we have that for any $r<\dist(x_0,\partial D)$ there holds
        \begin{itemize}
            \item $H(u,x_0,r)>0$;
            \item the map $\mathcal{N}(u,x_0,\cdot)$ is absolutely continuous in $(0,r)$;
            \item $\partial_r\, \mathcal{N}(u,x_0,r)\geq 0$.
        \end{itemize}
    \end{theorem} 
    
    Hence, the \emph{frequency} of $u$ 
	\begin{equation*}
		\gamma(u,x_0):=\lim_{r\to 0^+}\mathcal{N}(u,x_0,r)
	\end{equation*}
	is well-defined at any point $x_0\in D$. As a standard consequence, one can derive unique continuation properties implying that $\mathcal{F}(u)$ has empty interior and that
	\begin{equation*}
		\mathcal{F}(u)=D\cap\Big(\bigcup_{i=1}^N\partial\Omega_i^u\Big).
	\end{equation*}
    Hence, we can also split the points of the free boundary in terms of their frequency; namely, for any $\gamma>0$ we let
    \begin{equation*}
        \mathcal{F}_\gamma(u):=\{x\in\mathcal{F}(u)\colon \gamma(u,x)=\gamma\}.
    \end{equation*}
Finally, we point out that also the Weiss energy
\begin{equation*}
    W_\gamma(u,x_0,r):=\frac{H(u,x_0,r)}{r^{2\gamma}}\Big(\mathcal{N}(u,x_0,r)-\gamma\Big)
\end{equation*}
is monotone non-decreasing with respect to $r>0$, for any $0<\gamma\leq \gamma(u,x_0)$. This is a standard direct corollary of the monotonicity of $\mathcal{N}(u,x_0,\cdot)$ (for the proof we refer e.g. to \cite[Proposition 6.16]{OV}).

	\subsection{Blow-up analysis}\label{sub:intro-blow-up} 
    Let $u\in\mathcal{S}(D;N)$, respectively $\mathcal{M}(D;N)$, and let $x_0\in\mathcal F_\gamma(u)$, for some $\gamma>0$. We define the \emph{Almgren rescalings} as
	\[
	u_{x_0,r}(x):=\frac{u(x_0+rx)}{\sqrt{H(u,x_0,r)}}.
	\]
	As a consequence of the Almgren monotonicity formula, we have that $u_{x_0,r}$ converges, up to subsequences, to an element of 
    \begin{equation*}
		\mathcal{S}_\gamma(\R^d;N):=\left\{u\in\mathcal{S}(\R^d;N)\colon u~\text{is $\gamma$-homogeneous}\right\},
	\end{equation*}
	respectively,
	\begin{equation*}
		\mathcal{M}_\gamma(\R^d;N):=\left\{u\in\mathcal{M}(\R^d;N)\colon u~\text{is $\gamma$-homogeneous}\right\}.
	\end{equation*}
    More precisely, we have the following.
    \begin{theorem}
    For any sequence $r_n\to0$ there is a subsequence (still denoted by $r_n$) such that
	\begin{equation*}
		u_{x_0,r_n}\to U\quad\text{uniformly in $B_R$ and in }H^1(B_R;\R^N) 
	\end{equation*}
	as $n\to\infty$, for all $R>0$, for some $U\in\mathcal{S}_\gamma(\R^d;N)$, respectively $U\in\mathcal{M}_\gamma(\R^d;N)$.
    \end{theorem}
    For the proof, we refer e.g. to \cite[Theorem 3.3]{TerraciniTavares2012} and \cite[Proposition 6.13]{OV}.

    \subsection{Free boundary regularity}\label{sec:free_boundary_reg}
    We here summarize the main known results concerning the regularity and geometrical structure of the free boundary in the interior. This collects results from \cite{TerraciniTavares2012,ST-opt-gap}.
    \begin{theorem}
        For any $u\in\mathcal{S}(D;N)$, we have that 
    \begin{equation*}
        \text{either}\quad\gamma(u,x)=1\quad\text{or}\quad \gamma(u,x)\geq \frac{3}{2},
    \end{equation*}
    for any $x\in D$. Moreover, we can decompose
    \begin{equation*}
			\mathcal{F}(u)=D\cap\Big(\bigcup_{i=1}^N \partial \Omega_i^u\Big)=\mathrm{Reg}(u)\cup \mathrm{Sing}(u),
		\end{equation*}
		where 
        $$\mathrm{Reg}(u)=\{x\in D\colon \gamma(u,x)=1\}=\mathcal{F}_1(u)$$
        is the union of $(d-1)$-dimensional manifolds of class $C^{1,\alpha}$ and
        \begin{equation*}
            \mathrm{Sing}(u)=\left\{x\in D\colon \gamma(u,x)\geq \frac{3}{2}\right\}=\bigcup_{\gamma\geq \sfrac{3}{2}} \mathcal{F}_\gamma (u)
        \end{equation*}
        satisfies
        \begin{equation}\label{eq:dim_sing}
            \dim_{\mathcal{H}}\big(\mathrm{Sing}(u)\big)\leq d-2.
        \end{equation}
        For every $i=1,\dots,N$, the function 
        $u_i$ is differentiable at all points of $\overline\Omega_i^u\cap D$ and its gradient $\nabla u_i:\overline\Omega_i^u\cap D\to\R^d$ is continuous. Moreover, we have: 
        \begin{equation}\label{eq:grad_interface_sing}
        |\nabla u_i|=0\quad\text{on}\quad \mathrm{Sing}(u)\cap\partial\Omega_i\,,
        \end{equation}
        and 
		\begin{equation}\label{eq:grad_interface}
			|\nabla u_i|=|\nabla u_j|>0\quad\text{on}\quad\mathrm{Reg}(u)\cap \partial \Omega_i^u\cap\partial\Omega_j^u\,.
		\end{equation}   
        \end{theorem}

	\section{Singular blow-ups of minimal frequency}\label{sec:gap}
    The aim of this section is to classify homogeneous singular solutions with minimal frequency $\sfrac{3}{2}$, i.e. functions in $\mathcal{S}_{\sfrac{3}{2}}(\R^d;N)$.
	\subsection{A topological lemma}
	This section is dedicated to a topological lemma, which we use both in the classification of the $\sfrac32$-homogeneous blow-ups (\Cref{prop:32-homogeneous-classification}) and in the proof of the no-holes lemma (\Cref{l:no-holes}).

	\begin{definition}\label{def:family-of-curves}
		Let $\mathcal M$ be a smooth $m$-dimensional manifold and let $\mathcal F\subset \mathcal M$ be a $C^1$ submanifold
		of dimension $m-1$ (that is,  for every $x_0\in \mathcal F$ there is an open neighborhood $\omega\subset\mathcal M$ of $x_0$ and a $C^1$-diffeomorphism $\Phi:\omega\to B_1\subset\R^d$ such that $\Phi(\mathcal F)=B_1\cap\{x_m=0\}$). For any $x,y\in\mathcal M\setminus\mathcal F$, we denote by $C(x\to y)$ the family of piecewise $C^1$ curves transversal to $\mathcal F$. Precisely, we say that $\ell\in\mathcal C(x\to y)$ if: 
		\begin{enumerate}
			\item $\ell\colon[0,1]\to\mathcal M$, $\ell(0)=x$ and $\ell(1)=y$;
			\item $\ell\in C([0,1];\mathcal M)$ and there are $k=k(\ell)$ points $0=T_0< T_1< \dots< T_k=1,$
			such that: $\ell\in  C^{1}((T_{j-1},T_j);\mathcal M)$ and $\ell'\neq0$ on $(T_{j-1},T_j)$ for every $j=1,\dots, k$, and
			$$\ell(T_j)\notin\mathcal F\quad\text{for every}\quad j={0},\dots, k;$$
			\item if $\ell(t)\in\mathcal F$ for some $t\in(0,1)$, then $\ell$ is transversal to $\mathcal F$ at $\ell(t)$, that is, $\ell'(t)\notin T_{\ell(t)}\mathcal F$, where $T_{\ell(t)}\mathcal F$ is the tangent space to $\mathcal F$ at $\ell(t)$.
		\end{enumerate}	
	\end{definition}
	
	\begin{lemma}
		Let $\mathcal M$ be a smooth manifold of dimension $m$. Suppose that $\mathcal F$ is an $(m-1)$-dimensional $C^1$ submanifold of $\mathcal M$ and suppose that $\mathcal F$ is a relatively closed subset of $\mathcal M$. Let $x,y\in\mathcal M\setminus \mathcal F$ and let $\ell\in\mathcal C(x\to y)$ be a curve in $\mathcal M$. Then, the set of intersection points $\mathcal I(\ell):=\{t\in(0,1)\ :\ \ell(t)\in\mathcal F\}$ is finite. 
	\end{lemma}
	\begin{proof}
		Suppose that the set $\mathcal I$ contains an infinite sequence $(t_n)_{n\ge 1}$. Then, we can suppose that $t_n$ converges to some $t_\infty\in (0,1)$. By continuity of $\gamma$ and the fact that $\mathcal F$ is closed, we have $\gamma(t_\infty)\in\mathcal I$. But $\gamma$ is transversal to $\mathcal F$ at $t_\infty$. Thus, for $t\neq t_\infty$ in a neighborhood of $t_\infty$, $\gamma(t)\notin\mathcal F$, which is a contradiction.
	\end{proof}
	
	\begin{lemma}\label{l:topological-lemma}
		Let $\mathcal M$ be a simply connected smooth manifold of dimension $m$. Suppose that $\mathcal F\subset \mathcal M$ is a relatively closed subset of $\mathcal M$ and a $C^1$-submanifold of dimension $m-1$.
		Suppose that $x\in\mathcal M\setminus\mathcal F$ and that $\ell:[0,1]\to\mathcal M$ is a closed curve in $\mathcal C(x\to x)$.
		Then the set $\mathcal I(\ell)$ has an even number of elements.    
	\end{lemma}
	\begin{proof}
		Since $\mathcal M$ is simply connected, we can find a continuous homotopy $\Gamma:[0,1]\times[0,1]\to\mathcal M$ that deforms $\ell$ into the constant curve $x$. Precisely,
		$$\Gamma(0,\cdot)=\ell(\cdot)\qquad\text{and}\qquad\Gamma(1,\cdot)\equiv x.$$
		We proceed in several steps.    
		
		\noindent{\it Step 1. Discrete homotopy.} Since $\mathcal F$ is a $C^1$ embedded submanifold, we can find points $(s_i,t_j)$ of the form 
		$$(s_i,t_j)=\left(\frac{i}{n},\frac{j}{n}\right)\quad i=0,\dots,n\ ,\ j=1,\dots,n,$$
		such that for every set of \enquote{consecutive} 4 points:
		$$Q_{i,j}:=\Big\{X_{i,j}:=\Gamma(s_i,t_j),~X_{i+1,j}:=\Gamma(s_{i+1},t_j),~X_{i,j+1}:=\Gamma(s_i,t_{j+1}),~X_{i+1,j+1}:=\Gamma(s_{i+1},t_{j+1})\Big\}$$
		we have that either: 
		\begin{enumerate}
			\item[(Q1)] $Q_{i,j}$ is contained in an open set $\omega_1\subset \mathcal M\setminus\mathcal F$ diffeomorphic to a ball;
			\item[(Q2)] $Q_{i,j}$ is contained in an open set $\omega_2\subset\mathcal M$ for which there is a $C^1$-diffeomorphism $\Phi:\omega_2\to B_1\subset\R^m$ such that $\Phi(\mathcal F)=B_1\cap\{x_m=0\}$. 
		\end{enumerate}
		We next notice that by perturbing slightly each of the points $X_{i,j}$ we can construct another family of $n\times n$ points
		$$\Big\{\widetilde X_{i,j}\in\mathcal U\setminus\mathcal F\ :\ 0\le i\le n-1,\ 0\le j\le n-1\Big\},$$
		such that any set of the form
		$$\widetilde Q_{i,j}:=\Big\{\widetilde X_{i,j},\ \widetilde X_{i+1,j},\ \widetilde X_{i,j+1},\ \widetilde X_{i+1,j+1}\Big\}\,,$$
		satisfies (Q1) or (Q2) above. Moreover, we can also suppose that 
		$$\widetilde X(0,j)\in\ell([0,1])\quad\text{for every}\quad j=0,\dots,n;$$
		$$\widetilde X(n,j)=x\quad\text{for every}\quad j=0,\dots,n.$$
		
		\noindent{\it Step 2. Construction of a piecewise $C^1$-regular net with vertices $\widetilde X_{ij}$.} In what follows we will use the following notation, for every $0\le i,j\le n$
		\begin{itemize}
			\item $\omega_{ij}$ is the open set 
			$$\omega_{ij}:=\begin{cases}
				\omega_1\quad\text{if}\  (Q1)\ \text{is verified for the set}\ \widetilde Q_{ij}\,;\\
				\omega_2\quad\text{if}\ (Q2)\ \text{is verified for the set}\  \widetilde Q_{ij}\,.
			\end{cases}	$$	
			\item $\alpha_{ij}\in\mathcal C(\widetilde X_{i,j}\to\widetilde X_{i,j+1})$	is a curve lying in $\omega_{ij}$ (so $\alpha_{ij}$ does not cross $\mathcal F$);
			\item $\beta_{ij}^+\in\mathcal C(\widetilde X_{i,j}\to\widetilde X_{i+1,j})$	is a curve lying in $\omega_{ij}$ and $\beta_{ij}^-(t):=\beta_{ij}^+(1-t)$.
		\end{itemize}	
		Finally, for every $i=1,\dots,n$, we set 
		$$\ell_i:=\alpha_{i1}\ast\alpha_{i2}\ast\dots\ast\alpha_{i(n-1)},$$
		where $\ast$ denotes the usual concatenation of curves.
		We notice that, up to a reparametrization, we can choose the curves  $\alpha_{0j}$, $j=1,\dots,n$, to be pieces of $\ell$. Moreover, we can choose $\alpha_{nj}$, $j=1,\dots,n$, to be the constant curves $\alpha_{nj}(t)=x$ for all $t$. Thus, for all $t$, we have:
		$$\ell_0(t)=\ell(t)\qquad\text{and}\qquad\ell_n(t)=x.$$
		\noindent{\it Step 3. The map $M$.} In what follows, for any curve $\ell:[0,1]\to\mathcal M$ with $\ell\in\mathcal C(x\to y)$ for some $x,y\in\mathcal M\setminus\mathcal F$, we will denote by $M(\ell)$ the number of elements of $\mathcal I(\ell)$. The map $M$ has the following properties: 
		\begin{enumerate}
			\item[(M1)] Let $x,y,z\in \mathcal M\setminus\mathcal F$. If $\ell_1\in\mathcal C(x\to y)$ and $\ell_2\in\mathcal C(y\to z)$, then $\ell_1\ast \ell_2\in\mathcal C(x\to z)$ and 
			$$M(\ell_1\ast\ell_2)=M(\ell_1)+M(\ell_2).$$
			\item[(M2)]  Let $x,y\in \mathcal M\setminus\mathcal F$. If $\ell_+\in\mathcal C(x\to y)$ and $\ell_-$ is defined as 
			$$\ell_-(t)=\ell(1-t),$$
			then $\ell_-\in \mathcal C(y\to x)$ and $M(\ell_+)=M(\ell_-)$.
			\item[(M3)]  Let $x,y,z,w\in \mathcal M\setminus\mathcal F$ be 4 points in a set $\omega_2$ as in (Q2) and let the curves $\alpha,\beta,\gamma,\delta$ be as follows: 
			$$\alpha\in \mathcal C(x\to y)\ ,\quad \beta\in \mathcal C(y\to z)\ ,\quad \gamma\in \mathcal C(z\to w)\ ,\quad \delta\in \mathcal C(x\to w)\ .$$
			Then, we have that 
			$$M(\delta)=M(\alpha\ast\beta\ast\gamma)=M(\alpha)+M(\beta)+M(\gamma).$$
		\end{enumerate}	
		
		\noindent{\it Step 4. Conclusion.} Now, using the properties (M1), (M2) and (M3), we obtain that, for every pair of consecutive indices $i$ and $i+1$ with $0\le i<i+1\le n$, the following holds:
		\begin{align*}
			M(\ell_{i+1})=\sum_{j=0}^{n-1}M(\alpha_{(i+1)j})
			&=	\sum_{j=0}^{n-1}M(\beta_{ij}^-\ast \alpha_{ij}\ast\beta_{ij}^+)\\
			&=	\sum_{j=0}^{n-1}\Big(M(\beta_{ij}^-)+M(\alpha_{ij})+M(\beta_{ij}^+)\Big)\\\
			&=	\sum_{j=0}^{n-1}\Big(M(\alpha_{ij})+2M(\beta_{ij}^+)\Big)\\
			&=M(\ell_{i})+2\sum_{j=0}^{n-1}M(\beta_{ij}^+),
		\end{align*}	
		so that 
		$$(-1)^{M(\ell_{i+1})}=(-1)^{M(\ell_i)}.$$
		Since $i$ is arbitrary, this implies 
		$$(-1)^{M(\ell_n)}=(-1)^{M(\ell_0)}=(-1)^{M(\ell)}.$$
		Finally, since $\ell_n$ is the constant curve $\ell_n(t)=x$ for every $t$, we get that $M(\ell_n)=0$. This concludes the proof since $M(\ell)$ by definition is the number of elements of $\mathcal I(\ell)$.
	\end{proof}
	
	\subsection{\texorpdfstring{Classification of the blow-ups of frequency $\sfrac{3}{2}$}{Classification of the blow-ups of frequency 3/2}}\label{sub:classification}	%
		In this section we prove that any $\sfrac32$-homogeneous function $u\in \mathcal S_{\sfrac{3}{2}}(\R^d;N)$ is of the form $Y$ (see \cref{prop:32-homogeneous-classification}).  We stress that even if this result is not explicitly stated in  \cite{ST-opt-gap}, the key elements (the dimension reduction argument and \cref{l:optimal-gap-main-lemma}) are already contained in the proof of \cite[Lemma 4.2]{ST-opt-gap}; here below we give the details since we need \cref{prop:32-homogeneous-classification} in \cref{t:main-main}.

	\begin{lemma}\label{l:optimal-gap-main-lemma}
		Let $d\ge 3$ and $N\ge 2$. Suppose that $u=(u_1,\dots,u_N)\in\mathcal S_\gamma(\R^d;N)$ is a non-trivial $\gamma$-homogeneous function.
	Suppose that all the free boundary points on the unit sphere are regular, that is
		\[
		\mathcal F(u)\equiv \mathcal F_1(u)\quad\text{on}\quad \partial B_1,
		\]
		and that all the sets $\Omega_i^u\cap\partial B_1$, $i=1,\dots,N$ are connected. Then, there is a function 
		\[
		\sigma:\{1,\dots,N\}\to\{-1,1\},
		\]
		such that the function $\sigma\cdot u$, given by
		\[
		(\sigma\cdot u)(x):=\sum_{i=1}^N\sigma(i)u_i(x)
		\]
		is harmonic in $\R^d$. In particular, $\gamma$ is an integer. 
	\end{lemma}
	\begin{proof}
		Without loss of generality we can suppose that 
		$$\Omega_i^u\cap \partial B_1\neq \emptyset\quad\text{for every}\quad i=1,\dots,N.$$			
		We define the function $\sigma$ as follows. 
		We set $\sigma(1)=1$ and we fix a point $x_0\in \Omega_1^u\cap B_r$. For every point $x_1\in\partial B_1\setminus\mathcal F(u)$ we take a curve $\ell\in \mathcal C(x_0\to x_1)$ (where $\mathcal C(x_0\to x_1)$ is as in \cref{def:family-of-curves} with $\mathcal M=\partial B_1$ and $\mathcal F=\mathcal F(u)$) and we define
		$$\sigma(x_1):=(-1)^{M(\ell)},$$
		$M(\ell)$ being the number of times $\ell$ crosses $\mathcal F(u)$:
		\begin{equation}\label{e:definition-of-M}
			M(\ell):=\sharp\Big\{t\in[0,1]\ :\ \ell(t)\in\mathcal F(u)\Big\}.
		\end{equation}
		By \cref{l:topological-lemma} (applied to the sphere $\mathcal M=\partial B_1$ and $\mathcal F=\mathcal F(u)$) the value of $\sigma(x_1)$ does not depend on the choice of the curve $\ell$. In order to prove the harmonicity of $\sigma\cdot u$ we notice that if $B_r(x_0)$ is a ball such that $\mathcal F(u)\cap B_r(x_0)=\partial\Omega_i\cap \partial\Omega_j\cap B_r(x_0)$ is a smooth graph in $B_r(x_0)$, then for every pair of points $x_i\in\Omega_i\cap B_r(x_0)$ and $x_j\in\Omega_j\cap B_r(x_0)$, we can find a curve in $\mathcal C(x_i\to x_j)$ that crosses $\mathcal F(u)$ only once. Thus, by construction  $\sigma(x_i)=-\sigma(x_j)$
		and so the function $\sigma\cdot u$ is harmonic across $\mathcal F(u)$. This concludes the proof of \cref{l:optimal-gap-main-lemma}.
	\end{proof}
	
	\begin{proposition}\label{prop:32-homogeneous-classification}
	Let $u\in \mathcal S_{\sfrac32}(\R^d;N)$. Then, up to relabeling the components of $u$, we have that $u_j\equiv 0$ for $j\ge 4$, and, up to a rotation of the coordinate system, $(u_1,u_2,u_3)$ is of the form $(u_1,u_2,u_3)=cY$ for some constant $c\in\R$ and $Y=(Y_1,Y_2,Y_3)$ as in \eqref{eq:def-Y}.
	\end{proposition}	
	\begin{proof}
	 Let $u\in\mathcal{S}_{\sfrac{3}{2}}(\R^d;N)$. By \Cref{l:optimal-gap-main-lemma}, there is a point $x_0\in \left(\mathcal F(u)\setminus\mathcal F_1(u)\right)\cap\partial B_1$. But then, the upper semicontinuity of $\gamma(u,\cdot)$ and the homogeneity of $u$ give
	$$\frac32=\gamma(u,0)\ge \lim_{t\to+\infty}\gamma(u,tx_0)=\gamma(u,x_0).$$
	Since, by \cite{ST-opt-gap}, $\sfrac32$ is the minimal frequency, we get that 
	$$\frac32=\gamma(u,0)=\gamma(u,x_0),$$ 
which implies that $u$ is invariant in the direction of $x_0$. Up to a rotation of the coordinate system, we can suppose that $x_0=\vec{\bm{e}}_1$. Then $\widetilde u(x_2,\dots, x_d):=u(0,x_2,\dots, x_d)$ is a $\sfrac{3}{2}$-homogeneous global solution in $\R^{d-1}$, i.e. it belongs to $\mathcal{S}_{\sfrac{3}{2}}(\R^{d-1};N)$. By repeating this argument in every dimension $d-k$, $k\ge 1$, we finally get that, up to a rotation of the coordinate axes, we have
	$$u(x_1,\dots, x_{d-1},x_d)=w(x_{d-1},x_d),$$ 
	where $w$ is a $\sfrac{3}{2}$-homogeneous solution in dimension 2. This gives the specific form of $w$ and the fact that $w$ has exactly three non-zero components.	
\end{proof}	

We notice that as a consequence of \Cref{prop:32-homogeneous-classification}, we obtain the uniqueness of the minimizer (in any dimension) to the optimal 3-partition of the sphere for the min-max problem studied in \cite{CTVFucick2005, HHOT09,HHOT10,ST-opt-gap}.
\begin{corollary}\label{cor:min-max}
The $Y$-configuration is the unique minimizer of
\begin{equation*}
		\mathcal L_3:=	\min\left\{ \max_{i\in\{1,2,3\}}\lambda_1(\omega_i)\colon \omega_i\sub\partial B_1~\text{open, connected and s.t. }\omega_i\cap\omega_j=\emptyset ~\text{for }i\neq j \right\},
\end{equation*}
where $\lambda_1(\omega_i)$ is the first eigenvalue of the spherical Dirichlet Laplacian on $\omega_i$.
\end{corollary}	
\begin{proof}
Let $(\omega_1,\omega_2,\omega_3)$ be a triple achieving the minimum (see for instance \cite{HHOT09}). We know that $\omega_i\subset\partial B_1$ is measurable and $\mathcal H^{d-1}(\omega_i)>0$, for every $i$. 	Moreover, if $\phi_i$ is a first normalized  eigenfunction on $\omega_i$, then by \cite[Theorem 3.4]{HHOT09}, there are constants $a_i\ge 0$, $i=1,2,3$, such that $\psi:=(a_1\phi_1,a_2\phi_2,a_3\phi_3)$ satisfies
	\begin{equation*}
	\begin{bvp}
		-\Delta \psi_i &\leq 	\mathcal L_3\psi_i,&&\text{in }\partial B_1, \\
		-\Delta \Big(\psi_i-\sum_{j\neq i}\psi_j\Big)&\geq 	\mathcal L_3\Big(\psi_i-\sum_{j\neq i}\psi_j\Big), &&\text{in }\partial B_1.
	\end{bvp}
\end{equation*} 
By \cite[Theorem 3.4]{HHOT09} that $(a_1,a_2,a_3)\neq(0,0,0)$. We now claim that $a_i>0$ for all $i$. Suppose by contradiction that $a_3=0$, so that also $\psi_3=0$. By, the equation above, we get that $-\Delta (\psi_1-\psi_2)=\mathcal L_3\,(\psi_1-\psi_2)$ on $\partial B_1$. By the unique continuation principle for eigenfunctions, we get that $\mathcal H^{d-1}(\{\psi_1-\psi_2=0\})=0$, which leads to a contradiction since $\mathcal H^{d-1}(\omega_3)>0$. Now, in view of  \cite[Remark 3.14 (d)]{HHOT09} and \cite[Theorem 1.10]{ST-opt-gap}, we know that $\mathcal L_3=\lambda_1(\omega_i)=\frac32\left(\frac32+d-2\right)$ for all $i$. Thus, the $\sfrac32$ homogeneous extension $u$ of $\psi$ belongs to the class $\mathcal S_{\sfrac32}(\R^d;3)$, so the conclusion follows from \Cref{prop:32-homogeneous-classification}.
\end{proof}

	\section{Linearization around triple junctions}\label{sec:linearized}
	We now introduce the linearized problem at points of frequency $\sfrac{3}{2}$, which plays a crucial role in the proof of the epiperimetric inequality.
	For $i=1,2,3$, we set 
	\begin{equation*}
		\Omega_i:=\left\{(r\cos\theta,r\sin\theta)\colon r> 0~\text{and } \theta\in\left(-\pi+(i-1)\frac{2\pi}{3},-\pi+i\frac{2\pi}{3}\right)\right\},
	\end{equation*}
	so that $\R^2$ is the disjoint union 
	$$\R^2=\Omega_1\cup\Omega_2\cup\Omega_3\cup \{Y=0\}.$$
	Moreover, in $\R^d$ we may identify $\Omega_i$ with its cylindrical extension $\R^{d-2}\times\Omega_i$. 
	Now, the linearized problem around the triple junction solution $Y=(Y_1,Y_2,Y_3)\in\mathcal{S}_{\sfrac{3}{2}}(\R^2;3)$, which we recall to be defined as
	\begin{equation*}
		Y_i(r,\theta)=\begin{cases}
			r^{\frac{3}{2}}\abs{\cos\left(\frac{3}{2}\theta\right)},&\text{for }-\pi+\frac{2\pi}{3}(i-1)\leq \theta\leq -\pi+\frac{2\pi}{3}i, \\
			0, &\text{elsewhere},
		\end{cases}
	\end{equation*}
	is the following:
	\begin{equation}\label{eq:linearized}
		\begin{bvp}
			-\Delta w_k&=0, &&\text{in }\Omega_k~\text{for every }k\in\{1,2,3\},\\
			w_i&=-w_j,&&\text{on }\partial\Omega_i\cap\partial\Omega_j,~ \text{for every }i\neq j\in\{1,2,3\},\\
			\nabla w_i&=-\nabla w_j&&\text{on } \partial\Omega_i\cap\partial\Omega_j,~ \text{for every } i\neq j\in\{1,2,3\}.
		\end{bvp}
	\end{equation}
	More precisely, we say that $w=(w_1,w_2,w_3)$ is a solution of \eqref{eq:linearized} if 
	\begin{equation*}
		\begin{cases}
			w_i-w_j\in H^1_{\tu{loc}}(\Omega_{ij}), \\
			-\Delta(w_i-w_j)=0\quad\text{in }\Omega_{ij},
		\end{cases}
	\end{equation*}
	where
	\begin{equation*}
		\Omega_{ij}:=\Omega_i\cup\Omega_j\cup\left(\partial\Omega_i\cap\partial\Omega_j\right)\setminus\{0\}
	\end{equation*}
	for all $i, j\in\{1,2,3\}$, $i\neq j$.
	
	We remark that the linearized problem can also be rephrased with the cut space as domain, drawing a parallel with the thin obstacle problem. Namely, if we let
	\begin{equation*}
		\mathcal{P}:=\{(x'',x_{d-1},x_d)\in\R^d\colon x_{d-1}\leq 0,~x_d=0\}
	\end{equation*}
	and $\widetilde{w}:=-w_1+w_2-w_3$, then we have the following:
	\begin{itemize}
		\item $\widetilde{w}\in H^1_\tu{loc}(\R^d\setminus\mathcal{P})$ solves
		\begin{equation*}
			\begin{bvp}
				-\Delta w&=0, &&\text{in }\R^d\setminus\mathcal{P}, \\
				\lim_{x_d\to 0^+}w+\lim_{x_d\to 0^-}w&=0, &&\text{on }\mathcal{P}, \\
				\lim_{x_d\to 0^+}\partial_{x_d}w+\lim_{x_d\to 0^-}\partial_{x_d} w&=0, &&\text{on }\mathcal{P};
			\end{bvp}
		\end{equation*}
		\item we can write $\widetilde{w}=w_\tu{e}+w_{\tu{o}}$, where $w_{\tu{e}}$ and $w_{\tu{o}}$ are, respectively, even and odd with respect to $x_d$:
		\begin{equation*}
			w_{\tu{e}}(x',x_d):=\frac{w(x',x_d)+w(x',-x_d)}{2},\qquad			w_{\tu{o}}(x',x_d):=\frac{w(x',x_d)-w(x',-x_d)}{2};
		\end{equation*}
		\item the even part $w_{\tu{e}}$ solves
		\begin{equation*}
			\begin{bvp}
				-\Delta w_{\tu{e}}&=0, &&\text{in }\R^d\setminus\mathcal{P}, \\
				w_{\tu{e}}&=0, &&\text{on }\mathcal{P};
			\end{bvp}
		\end{equation*}
		\item the odd part $w_{\tu{o}}$ solves
		\begin{equation*}
			\begin{bvp}
				-\Delta w_{\tu{o}}&=0, &&\text{in }\R^d\setminus\mathcal{P}, \\
				\partial_{x_d}w_{\tu{o}}&=0, &&\text{on }\mathcal{P}.
			\end{bvp}
		\end{equation*}
	\end{itemize}
	A key observation is that, if we consider the even extension of the restriction of $w_\tu{o}$ to $\R^d_+$ and exchange $x_{d-1}$ with $-x_{d-1}$, that is we let
	\begin{equation*}
		\bar{w}(x'',x_{d-1},x_d):=\begin{cases}
			w_\tu{o}(x'',-x_{d-1},x_d),&\text{for }x_d>0, \\
			w_\tu{o}(x'',-x_{d-1},-x_d),&\text{for }x_d\leq 0, 
		\end{cases}
	\end{equation*}
	then one can see that $\bar{w}$ solves the same equation as the even part $w_\tu{e}$, namely
	\begin{equation*}
		\begin{bvp}
			-\Delta\bar{w}&=0, &&\text{in }\R^d\setminus\mathcal{P}, \\
			\bar{w}&=0, &&\text{on }\mathcal{P}.
		\end{bvp}
	\end{equation*}
	On the other hand, this coincides with the linearized problem of the thin obstacle problem and so, for what concerns the classification of $\sfrac{3}{2}$-homogeneous solutions, we can appeal to the known results in this setting (see e.g. \cite{GPGthin,FSthin}). Therefore, we proved the following.
	\begin{proposition}\label{prop:classification}
		Let $\widetilde{w}\in H^1_{\textup{loc}}(\R^d\setminus\mathcal{P})$ be a $\sfrac{3}{2}$-homogeneous solution of 
		\begin{equation*}
			\begin{bvp}
				-\Delta w&=0, &&\text{in }\R^d\setminus\mathcal{P}, \\
				\lim_{x_d\to 0^+}w+\lim_{x_d\to 0^-}w&=0, &&\text{on }\mathcal{P}, \\
				\lim_{x_d\to 0^+}\partial_{x_d}w+\lim_{x_d\to 0^-}\partial_{x_d} w&=0, &&\text{on }\mathcal{P}.
			\end{bvp}
		\end{equation*}
		Then $\widetilde{w}=w_{\tu{e}}+w_{\tu{o}}$, where $w_{\tu{e}}$ is even with respect to $x_d$ and takes the form
		\begin{equation*}
			w_{\tu{e}}=a\widetilde{Y}(x_{d-1},x_d)+U_0(x_{d-1},x_d)\sum_{i=1}^{d-2}a_i x_i
		\end{equation*}
		and $w_{\tu{o}}$ is odd with respect to $x_d$ and takes the form
		\begin{equation*}
			w_{\tu{o}}=b\widetilde{Y}(-x_{d-1},x_d)+U_0(-x_{d-1},x_d)\sum_{i=1}^{d-2}b_i x_i\quad\text{for }x_d>0,
		\end{equation*}
		where
		\begin{equation*}
			\widetilde{Y}(r,\theta)=r^{\frac{3}{2}}\cos\left(\frac{3}{2}\theta\right)=-Y_1+Y_2-Y_3\quad\text{and}\quad 			U_0(r,\theta)=r^{\frac{1}{2}}\cos\left(\frac{\theta}{2}\right).
		\end{equation*}
		In particular, $w_{\tu{e}}$ and $w_{\tu{o}}$ are $\sfrac{3}{2}$-homogeneous solutions of
		\begin{equation*}
			\begin{bvp}
				-\Delta w_{\tu{e}}&=0, &&\text{in }\R^d\setminus\mathcal{P}, \\
				w_{\tu{e}}&=0, &&\text{on }\mathcal{P}
			\end{bvp}
			\qquad\text{and} \qquad
			\begin{bvp}
				-\Delta w_{\tu{o}}&=0, &&\text{in }\R^d\setminus\mathcal{P}, \\
				\partial_{x_d}w_{\tu{o}}&=0, &&\text{on }\mathcal{P}.
			\end{bvp}
		\end{equation*}
	\end{proposition}

	\section{\texorpdfstring{Epiperimetric inequality at points of frequency $\sfrac32$}{Epiperimetric inequality at points of frequency 3/2}}\label{sec:proof-of-epi}
	In this section we prove \cref{t:epi}; we start with the following general lemma, which will be useful in the construction of the competitors in Step 1 and Step 4 of the proof.
    
	\begin{lemma}[Distance on $\Sigma_N$]\label{lemma:dist}
		For any $u,v\in H^1(B_1;\Sigma_N)\cap C^{0,1}(B_1;\R^N)$ such that $u_i,v_i\geq 0$ in $B_1$ for all $i=1,\dots,N$, we have
		\begin{equation*}
			\d_{\Sigma_N}(u,v)=\Big|v_i-u_i+\sum_{j\neq i}u_j\Big|=\abs{v_i-\widehat{u}_i} \quad\text{in }\overline{\{v_i>0\}},
		\end{equation*}
            where
            \begin{equation}\label{eq:hat_op}
                \widehat{u}_i:=u_i-\sum_{j\neq i}u_j.
            \end{equation}
		Moreover, we also have
		\begin{equation*}
			\d_{\Sigma_N}(u,v)=\sum_{j=1}^N|u_j-v_j|\quad\text{in }B_1.
		\end{equation*}
	\end{lemma}
	\begin{proof}
		Let $x\in\overline{\{v_i>0\}}$. If $x\in \{u_i>0\}$, then 
		$$\d_{\Sigma_N}(u,v)(x)=|v_i(x)-u_i(x)|=\Big|v_i(x)-u_i(x)+\sum_{j\neq i}u_j\Big|=\sum_{j=1}^N|v_j(x)-u_j(x)|,$$
		while, if $x\in \overline{\{u_k>0\}}$, with $k\neq i$, then $$\d_{\Sigma_N}(u,v)(x)=v_i(x)+u_k(x)=\Big|v_i(x)-u_i(x)+\sum_{j\neq i}u_j(x)\Big|=\sum_{j=1}^N|v_j(x)-u_j(x)|.$$
		If $u_k(x)=0$ for all $k=1,\dots,N$ the statement is trivial.
	\end{proof}

    The following lemma contains a general integration by parts identity for tangent maps. 
    
	\begin{lemma}\label{lemma:int_by_parts}
		Suppose that $v\in \mathcal{S}_\gamma(\R^d;N)$. Then, for every $u\in C^{0,1}(B_1)$ and every $i=1,\dots,N$, we have that
		$$\int_{B_1\cap\{v_i>0\}}\nabla v_i\cdot\nabla u\dx=\gamma\int_{\partial B_1\cap\{v_i>0\}}v_i u\ds+\int_{B_1\cap\partial\{v_i>0\}}u\partial_{\nnu}v_i \ds,$$
		where $\nu$ is the normal to $\mathrm{Reg}(v)\cap\partial\{v_i>0\}$ pointing outwards $\{v_i>0\}$.
	\end{lemma}
	\begin{proof}
		Before starting, we first recall that we know some basic structure of the free boundary of $v$ (see \Cref{sec:free_boundary_reg}); more precisely, from \cite{TerraciniTavares2012} we know that $v\in C^{0,1}(B_1;\R^N)$ and that we can decompose
		\begin{equation*}
			\mathcal{F}(v)=\bigcup_{i=1}^N \partial \{v_i>0\}=\mathrm{Reg}(v)\cup \mathrm{Sing}(v),
		\end{equation*}
		where $\mathrm{Reg}(v)$ is the union of $(d-1)$-dimensional manifolds of class $C^{1,\alpha}$ and $\mathrm{Sing}(v)$ has zero $\mathcal{H}^{d-1}$ measure. Moreover, from \eqref{eq:grad_interface_sing} and \eqref{eq:grad_interface}, we have
		\begin{equation*}
			|\nabla v_i|=|\nabla v_j|\quad\text{on }\mathrm{Reg}(v)\qquad\text{and}\quad |\nabla v_i|=0\quad\text{on }\mathrm{Sing}(v).
		\end{equation*}
		Hence, for every $\eps>0$ we can find a finite family of balls $B_{r_k}(x_k)$ such that:
		$${\rm Sing}(v)\cap \overline B_1\subset\bigcup_kB_{r_k}(x_k)\qquad\text{and}\qquad \sum_kr_k^{d-1}\le \eps.$$  
		For every ball we consider a smooth function $\phi_k\in C^\infty_c(\R^d)$ such that: 
		$$\phi_k\equiv 1\ \text{ on }\ B_{r_k}(x_k),\qquad \phi_k\equiv 0\ \text{ on }\ \R^d\setminus B_{2r_k}(x_k),\qquad |\nabla \phi_k|\le 2r_k^{-1}.$$
		and we set $\phi=\phi_\eps:=\sup_k\phi_k$. Then, since $u(1-\phi)$ is identically zero in a neighborhood of $\mathrm{Sing}(v)$, we have that 
		$$\int_{B_1\cap\{v_i>0\}}\nabla \big(u(1-\phi)\big)\cdot\nabla v_i=\gamma\int_{\partial B_1\cap\{v_i>0\}}(1-\phi)v_i u\dS+\int_{B_1\cap\partial\{v_i>0\}}(1-\phi)u\partial_{\nnu}v_i \dS.$$
		Thus, in order to conclude the proof, it is sufficient to estimate the error term
		$$e:=\int_{B_1\cap\{v_i>0\}}(\phi\nabla u+u\nabla\phi)\cdot\nabla v_i+\gamma\int_{\partial B_1\cap\{v_i>0\}}\phi v_i u\dS+\int_{B_1\cap\partial\{v_i>0\}}\phi u\partial_{\nnu}v_i \dS.$$
		By the definition of $\phi$ we have that 
		\begin{align*}
			|e|&\le\sum_{k}\Big(\int_{B_{2r_k}(x_k)\cap B_1\cap \{v_i>0\}}|\phi_k||\nabla u||\nabla v_i|+\int_{B_{2r_k}(x_k)\cap B_1\cap \{v_i>0\}}|\nabla\phi_k||u||\nabla v_i|\\
			&\qquad+\gamma\int_{B_{2r_k}(x_k)\cap\partial B_1\cap\{v_i>0\}}|\phi_k| |v_i| |u|\dS+\int_{B_{2r_k}(x_k)\cap B_1\cap\partial\{v_i>0\}}|\phi_k| |u||\nabla v_i| \dS\Big) \\
			&\le C\sum_{k}\Big(r_k^{d}+r_k^dr_k^{-1}+\gamma r_k^{d-1}+r_k^{d-1}\Big)\le C\sum_{k}r_k^{d-1}\le C\eps, 
		\end{align*}
		where the constant $C$ depends on $d$, $\|v\|_{L^\infty}$, $\|\nabla v\|_{L^\infty}$, $\|u\|_{L^\infty}$, and $\|\nabla u\|_{L^\infty}$.
	\end{proof}

    The next lemma is the Weiss-energy linearization identity which is a key element in the argument in the proof of \cref{t:epi}.
    
	\begin{lemma}[Weiss-energy linearization identity]\label{lemma:weiss_diff}
		For any $u\in H^1(B_1;\Sigma_N)$ and $\gamma\geq 0$, let
		\begin{equation*}
			W_\gamma(u):=\sum_{i=1}^N\int_{B_1}|\nabla u_i|^2\dx-\gamma\sum_{i=1}^N\int_{\partial B_1}u_i^2\ds.
		\end{equation*}
		Then, for any $u\in H^1(B_1;\Sigma_N)\cap C^{0,1}(B_1;\R^N)$ and $v\in\mathcal{S}_\gamma(\R^d;N)$ such that $u_i,v_i\geq 0$ for all $i=1,\dots,N$, there holds
		\begin{equation*}
			W_{\gamma}(u)=W_{\gamma}(\d_{\Sigma_N}(u,v))+\beta(u,v),
		\end{equation*}
		where  
		\begin{equation}\label{eq:beta}
			\beta(u,v):=4\sum_{1\le i<j\le N}\int_{B_1\cap\partial\{v_i>0\}\cap\partial\{v_j>0\}}|\nabla v_i|\Big(\sum_{k\neq i,j}u_k\Big)\ds.
		\end{equation}	
	\end{lemma}
	\begin{proof}
		We prove the result by direct computations. In view of \Cref{lemma:dist} we have that
		\begin{align*}
			W_\gamma(\d_{\Sigma_N}(u,v))&=\sum_{i=1}^N\Bigg[\int_{B_1\cap\{v_i>0\}}\Big|\nabla\Big(v_i-u_i+\sum_{k\neq i}u_k\Big)\Big|^2\dx \\
			&\quad-\gamma\int_{\partial B_1\cap \{v_i>0\}}\Big|v_i-u_i+\sum_{k\neq i}u_k\Big|^2\ds\Bigg]
		\end{align*}
		and so, by explicit computations, we derive that
		\begin{align*}
			W_\gamma(\d_{\Sigma_N}(u,v))&=W_\gamma(u)+W_\gamma(v) \\
			\quad &-2\sum_{i=1}^N\int_{B_1\cap \{v_i>0\}}\nabla v_i\cdot\nabla\Big(u_i-\sum_{k\neq i}u_k\Big)\dx+2\gamma\int_{\partial B_1\cap\{v_i>0\}} v_i\Big(u_i-\sum_{k\neq i} u_k\Big)\ds.
		\end{align*}
		Now, in view of \Cref{lemma:int_by_parts}, \eqref{eq:grad_interface} and by homogeneity, we have that 
		\begin{align*}
			\sum_{i=1}^N\int_{B_1\cap \{v_i>0\}}\nabla v_i\cdot\nabla\Big(u_i-\sum_{k\neq i}u_k\Big)\dx&=\sum_{i=1}^N\int_{B_1\cap\partial \{v_i>0\}}\partial_{\nnu} v_i\Big(u_i-\sum_{k\neq i}u_k\Big)\dx \\
			&\quad +\sum_{i=1}^N\int_{\partial B_1}\partial_{\nnu}v_i\Big(u_i-\sum_{k\neq i}u_k\Big)\ds \\
			&=-\sum_{i=1}^N\int_{B_1\cap\partial \{v_i>0\}}|\nabla v_i|\Big(u_i-\sum_{k\neq i}u_k\Big)\dx \\
			&\quad +\gamma\sum_{i=1}^N\int_{\partial B_1}v_i\Big(u_i-\sum_{k\neq i}u_k\Big)\ds
		\end{align*}
		Therefore, since also $W_\gamma(v)=0$ (again from \Cref{lemma:int_by_parts}), we have
		\begin{align*}
			W_\gamma(\d_{\Sigma_N}(u,v))&=W_\gamma(u)+2\sum_{i=1}^N\int_{B_1\cap\partial \{v_i>0\}}|\nabla v_i|\Big(u_i-\sum_{k\neq i}u_k\Big)\dx \\
			&=W_\gamma(u)+2\sum_{i=1}^N\sum_{\substack{j=1 \\ j\neq i}}^N \int_{B_1\cap\partial \{v_i>0\}\cap\partial \{v_j>0\}}|\nabla v_i|\Big(u_i-\sum_{k\neq i}u_k\Big)\dx \\
			&=W_\gamma(u)-4\sum_{1\leq i<j\leq N}\int_{B_1\cap\partial \{v_i>0\}\cap\partial \{v_j>0\}}|\nabla v_i|\Big(\sum_{k\neq i,j}u_k\Big)\dx,
		\end{align*}
		where we used \eqref{eq:grad_interface_sing} and \eqref{eq:grad_interface}, thus concluding the proof.
	\end{proof}
	
	We conclude this section with the following.
	
	\begin{proof}[\bf Proof of \cref{t:epi}]
		In the whole proof, with an abuse of notation, for the sake of simplicity, we write $\Omega_i$ in place of $\Omega_i^Y\cap B_1$, for $i=1,2,3$.
		
		We reason by contradiction, and assume that there exist sequences $\{\delta_n\}_n\sub\R_+$, $\{\eps_n\}_n\sub\R_+$, $\{\tau_n\}_n\sub\R_+$, $\{N_n\}_n\sub\N$ and $\{c_n\}_n\sub H^1(B_1;\Sigma_{N_n})\cap C^{0,1}(B_1;\R^{N_n})$ such that 
		\begin{align}
			& c_{n,i}\geq 0\quad\text{for all }n\in\N~\text{and all }i=1,\dots,N_n, \notag\\
			&\delta_n,\eps_n,\tau_n\to 0\quad\text{as }n\to\infty, \notag \\
			&\sum_{i=1}^3\d_{\mathcal{H}}\left(\{c_{n,i}>0\},\{Y_i>0\}\right)+\sum_{i=4}^{N_n}\d_{\mathcal{H}}\left(\{c_i>0\},\{x_{d-1}=x_d=0\}\right)=\tau_n, \label{eq:epi5}\\
			&\norm{\d_{\Sigma_{N_n}}(c_n,Y)}_{H^1(B_1)}=\delta_n \notag
		\end{align}
		and
		\begin{equation}\label{eq:epi1}
			(1-\eps_n)W_{\frac{3}{2}}(c_n)<W_{\frac{3}{2}}(u)
		\end{equation}
		for all $u\in H^1(B_1;\Sigma_{N_n})\cap C^{0,1}(B_1;\R^{N_n})$ such that $u=c_n$ on $\partial B_1$. Moreover, without loss of generality, we can assume that
		\begin{equation}\label{eq:proj_L2}
			\norm{\d_{\Sigma_{N_n}}(c_n,Y)}_{L^2(B_1)}=\inf\left\{\norm{\d_{\Sigma_{N_n}}(c_n,h)}_{L^2(B_1)}\colon h\in \mathcal{M}_{\frac{3}{2}}(\R^d;3)\right\}.
		\end{equation}
        Let us prove this fact. Let $P c_n$ be the projection of $c_n$ onto the space of $3/2$-homogeneous blow-ups $\mathcal{M}_{\frac{3}{2}}(\R^d;3)$ with respect to the $L^2$-scalar product, i.e.
            \begin{equation}\label{eq:L2}
                \norm{d_{\Sigma_{N_n}}(c_n,Pc_n)}_{L^2(B_1)}=\inf\left\{\norm{\d_{\Sigma_{N_n}}(c_n,h)}_{L^2(B_1)}\colon h\in \mathcal{M}_{\frac{3}{2}}(\R^d;3)\right\}.
            \end{equation}
            We deduce that
            \begin{align*}
                \norm{d_{\Sigma_{N_n}}(Y,Pc_n)}_{L^2(B_1)}&\leq \norm{d_{\Sigma_{N_n}}(c_n,Pc_n)}_{L^2(B_1)}+\norm{d_{\Sigma_{N_n}}(c_n,Y)}_{L^2(B_1)} \\
                &\leq 2\norm{d_{\Sigma_{N_n}}(c_n,Y)}_{L^2(B_1)}\le2\delta_n.
            \end{align*}
            Hence, if $Pc_n=a_n Y\circ T_n$ for some $a_n>0$ and a rotation $T_n\in SO(d)$, we have that $a_n\to 1$ and $T_n\to I$ as $n\to\infty$, and this implies that
            \begin{equation*}
                \norm{\d_{\Sigma_{N_n}}(Y,Pc_n)}_{H^1(B_1)}\to 0\quad\text{as }n\to\infty.
            \end{equation*}
            Essentially, this is a consequence of the fact that, being $\mathcal{M}_{\frac{3}{2}}(\R^d;3)$ a finite dimensional manifold, the $L^2$ and the $H^1$ metrics are equivalent. We now observe that, since $T_n\in SO(d)$, we have
            \begin{align}
                \norm{\d_{\Sigma_{N_n}}\left(\frac{c_n\circ T_n^{-1}}{a_n},Y\right)}_{H^1(B_1)}&=\frac{1}{a_n}\norm{\d_{\Sigma_{N_n}}(c_n,Pc_n)}_{H^1(B_1)} \notag\\
                &\leq \frac{1}{a_n}\left(\norm{\d_{\Sigma_{N_n}}(c_n,Y)}_{H^1(B_1)} +\norm{\d_{\Sigma_{N_n}}(Y,Pc_n)}_{H^1(B_1)} \right)=:\delta_n'\to 0,\label{eq:delta'}
            \end{align}
            as $n\to\infty$. In addition, from \eqref{eq:L2}, we obtain that
            \begin{align}
                \norm{\d_{\Sigma_{N_n}}\left(\frac{c_n\circ T_n^{-1}}{a_n},Y\right)}_{L^2(B_1)} &=\inf\left\{\norm{\d_{\Sigma_{N_n}}\left(\frac{c_n\circ T_n^{-1}}{a_n},\frac{h\circ T_n^{-1}}{a_n}\right)}_{L^2(B_1)}\colon h\in \mathcal{M}_{\frac{3}{2}}(\R^d;3)\right\} \notag\\
                & =\inf\left\{\norm{\d_{\Sigma_{N_n}}\left(\frac{c_n\circ T_n^{-1}}{a_n},h\right)}_{L^2(B_1)}\colon h\in \mathcal{M}_{\frac{3}{2}}(\R^d;3)\right\}.\label{eq:L2'}
            \end{align}
            In particular, $Y$ is the projection of $\frac{c_n\circ T_n^{-1}}{a_n}$ onto the space $\mathcal{M}_{\frac{3}{2}}(\R^d;3)$ with respect to $L^2(B_1)$. Finally, one can observe that, by triangular inequality there holds
            \begin{align}
                &\sum_{i=1}^3\d_{\mathcal{H}}\left(\left\{\left(\frac{c_n\circ T_n^{-1}}{a_n}\right)_i>0\right\},\{Y_i>0\}\right)+\sum_{i=4}^{N_n}\d_{\mathcal{H}}\left(\left\{\left(\frac{c_n\circ T_n^{-1}}{a_n}\right)_i>0\right\},\mathrm{Sing}(Y)\right)\notag \\
                &\leq \sum_{i=1}^3\d_{\mathcal{H}}\left(\left\{(c_n\circ T_n^{-1})_i>0\right\},\{(Y\circ T_n^{-1})_i>0\}\right)\notag \\
                &+\sum_{i=4}^{N_n}\d_{\mathcal{H}}\left(\left\{(c_n\circ T_n^{-1})_i>0\right\},\mathrm{Sing}(Y\circ T_n^{-1})\right) \notag\\
                &+ \sum_{i=1}^3\d_{\mathcal{H}}\left(\{(Y\circ T_n^{-1})_i>0\},\{Y_i>0\}\right)+\sum_{i=4}^{N_n}\d_{\mathcal{H}}\left(\mathrm{Sing}\,(Y\circ T_n^{-1}),\mathrm{Sing}(Y)\right):=\tau_n'\label{eq:tau'}
            \end{align}
            and $\tau_n'\to 0$ as $n\to\infty$. Now, we rename
            \begin{equation*}
                \frac{c_n\circ T_n^{-1}}{a_n}~\text{as }c_n,\qquad \norm{\d_{\Sigma_{N_n}}\left(\frac{c_n\circ T_n^{-1}}{a_n},Y\right)}_{H^1(B_1)}~\text{as }\delta_n,
            \end{equation*}
            and
            \begin{equation*}
                \sum_{i=1}^3\d_{\mathcal{H}}\left(\left\{\left(\frac{c_n\circ T_n^{-1}}{a_n}\right)_i>0\right\},\{Y_i>0\}\right) +\sum_{i=4}^{N_n}\d_{\mathcal{H}}\left(\left\{\left(\frac{c_n\circ T_n^{-1}}{a_n}\right)_i>0\right\},\mathrm{Sing}(Y)\right)
            \end{equation*}
            as $\tau_n$, so that $\delta_n\leq \delta_n'$ and $ \tau_n\leq \tau_n'$, with $\delta_n',\tau_n'\to 0$ as in \eqref{eq:delta'} and \eqref{eq:tau'}, and we can proceed in the proof by assuming \eqref{eq:proj_L2}.

		Now, first of all, we apply \Cref{lemma:weiss_diff} with $v=Y$  and divide both sides by $\delta_n^2$, thus obtaining that
		\begin{equation*}
			(1-\eps_n)\left(W_{\frac{3}{2}}(\xi_n)+\frac1{\delta_n^2}\beta(c_n,Y)\right)<W_{\frac{3}{2}}\left(\frac{\d_{\Sigma_{N_n}}(u,Y)}{\delta_n}\right)+\frac1{\delta_n^2}\beta(u,Y),
		\end{equation*}
		for all $u\in H^1(B_1;\Sigma_{N_n})\cap C^{0,1}(B_1;\R^{N_n})$ such that $u=c_n$ on $\partial B_1$, where $\xi_n:=\d_{\Sigma_{N_n}}(c_n,Y)/\delta_n$. In other words,
		\begin{equation*}
			(1-\eps_n)\left(\int_{B_1}|\nabla \xi_n|^2\dx+\frac1{\delta_n^2}\beta(c_n,Y)\right)<\int_{B_1}\left|\nabla \Big(\frac{\d_{\Sigma_{N_n}}(u,Y)}{\delta_n}\Big)\right|^2\dx+\frac1{\delta_n^2}\beta(u,Y)+\frac{3\eps_n}{2}\int_{\partial B_1}\xi_n^2\ds.
		\end{equation*}
		We also let 
		\begin{equation*}
			w_{n,i}:=\begin{cases}
				\displaystyle\frac{Y_i-c_{n,i}}{\delta_n}+\sum_{j\neq i}\frac{c_{n,j}}{\delta_n},&\text{in }\Omega_i, \\
				0, &\text{in }B_1\setminus\Omega_i,
			\end{cases}
		\end{equation*}
		for $i=1,2,3$. Since
		\begin{equation*}
			\abs{w_{n,i}}=\frac{\d_{\Sigma_{N_n}}(c_n,Y)}{\delta_n}=\xi_n,\quad\text{in }\Omega_i,
		\end{equation*}
		see \Cref{lemma:dist}, in particular,
		\begin{equation*}
			\sum_{i=1}^3\norm{w_{n,i}}_{H^1(\Omega_i)}^2=\norm{\xi_n}_{H^1(B_1)}^2=1.
		\end{equation*}
		Therefore, there exist functions $w_i\in H^1(\Omega_i)$ such that, up to a subsequence,
		\begin{equation*}
			w_{n,i}\weak w_i\quad\text{weakly in }H^1(\Omega_i),~\text{as }n\to\infty,
		\end{equation*}
		for $i=1,2,3$. Hence, we can rephrase the contradiction assumption as
		\begin{multline}\label{eq:epi2}
			(1-\eps_n)\left(\sum_{i=1}^3\int_{\Omega_i}|\nabla w_{n,i}|^2\dx+\frac1{\delta_n^2}\beta(c_n,Y)\right) \\
			<\int_{B_1}\Big|\nabla \Big(\frac{\d_{\Sigma_{N_n}}(u,Y)}{\delta_n}\Big)\Big|^2\dx+\frac1{\delta_n^2}\beta(u,Y)+o(1),
		\end{multline}
		as $n\to\infty$,
		for all $u\in H^1(B_1;\Sigma_{N_n})\cap C^{0,1}(B_1;\R^{N_n})$ such that $u=c_n$ on $\partial B_1$, where the reminder term $o(1)$ is independent from $u$.
		We now proceed in the following steps, which, at the end, lead to a contradiction:
		\begin{itemize}
			\item \textbf{Step 1:} $\delta_n^{-2}\beta(c_n,Y)$ is uniformly bounded with respect to $n$;
			\item \textbf{Step 2:} $w=(w_1,w_2,w_3)$ is a solution of the linearized problem, see \Cref{sec:linearized};
			\item \textbf{Step 3:} $w_i=0$ for all $i=1,2,3$;
			\item \textbf{Step 4:} $w_{n,i}\to 0$ strongly in $H^1(\Omega_i)$, as $n\to\infty$, for all $i=1,2,3$.
		\end{itemize}
		\noindent \underline{Proof of \textbf{Step 1}.} For any $n\in\N$, we define a suitable \enquote{competitor} $u_n\in H^1(B_1;\Sigma_{N_n})$ such that $u_n=c_n$ on $\partial B_1$ as follows. First, in an inner ball, we let $u_n=Y$, namely
		\begin{equation*}
			u_n=Y\quad\text{in }B_{\frac{1}{2}}.
		\end{equation*}
		In the outer annulus $B_1\setminus B_{\frac{1}{2}}$, we consider the interpolation between $Y$ and $c_n$; namely, we first set
		\begin{equation*}
			t=t_r:=2\left(r-\frac{1}{2}\right),\quad\text{with }r=|x|.
		\end{equation*}
		Now, we let
		\begin{equation*}
			u_n(r,\theta)=r^{\frac{3}{2}}\left(t\,c_{n,i}(1,\theta)+(1-t)Y_i(1,\theta)\right)\,\vec{\bm{e}}_i
		\end{equation*}
		if $1/2<r<1$ and $Y_j(\theta)=c_{n,j}(1,\theta)=0$ for all $j\neq i$ and, on the other hand, we let
		\begin{equation*}
			u_n(r,\theta)=\begin{cases}
				r^{\frac{3}{2}}(Y_i(1,\theta)-t(Y_i(1,\theta)+c_{n,j}(1,\theta)))\,\vec{\bm{e}}_i,&\text{for }0<t<\frac{Y_i(1,\theta)}{Y_i(1,\theta)+c_{n,j}(1,\theta)}, \\
				r^{\frac{3}{2}}(-Y_i(1,\theta)+t(Y_i(1,\theta)+c_{n,j}(1,\theta)))\,\vec{\bm{e}}_j,&\text{for }\frac{Y_i(1,\theta)}{Y_i(1,\theta)+c_{n,j}(1,\theta)}\leq t<1,
			\end{cases}
		\end{equation*}
		if $1/2<r<1$ and $Y_i(\theta)>0$ and $c_{n,j}(1,\theta)>0$ for some $j\neq i$. If we now let
		\begin{equation*}
			\bar{u}_{n,i}:=\frac{Y_i-u_{n,i}+\sum_{j\neq i}u_{n,j}}{\delta_n}\quad\text{in }\Omega_i,
		\end{equation*}
		then there holds
		\begin{equation*}
			\abs{\bar{u}_{n,i}}=\frac{\d_{\Sigma_{N_n}}(u_n,Y)}{\delta_n}\quad\text{in }\Omega_i,
		\end{equation*}
		see \Cref{lemma:dist}, and one can easily check that
		\begin{equation*}
			\bar{u}_{n,i}=tw_{n,i}\quad\text{in }\Omega_i\cap\left(B_1\setminus B_{\frac{1}{2}}\right)\quad\text{and}\quad \bar{u}_{n,i}=0\quad\text{in }B_{\frac{1}{2}}.
		\end{equation*}
		Moreover, we have that
		\begin{equation*}
			u_n=tc_n\quad\text{on }\{(Y_1,Y_2,Y_3)=(0,0,0)\}.
		\end{equation*}
		Therefore, from \eqref{eq:epi2} we derive that
		\begin{equation}\label{eq:epi_beta_bdd}
			\frac{\beta((1-\eps_n-t)c_n,Y)}{\delta_n^2}<(\eps_n-1)\sum_{i=1}^3\int_{\Omega_i}|\nabla w_{n,i}|^2\dx+\sum_{i=1}^3\int_{\Omega_i}|\nabla (tw_{n,i})|^2\dx+o(1),
		\end{equation}
		as $n\to\infty$. On one hand, being $w_{n,i}$ bounded in $H^1(\Omega_i)$, the right hand side is uniformly bounded with respect to $n\in\N$, while, on the other hand, by direct computations we have
		\begin{equation}\label{eq:epi_beta_diff}
			\beta((1-\eps_n-t)c_n,Y)=(d+1)\int_0^1r^d(1-\eps_n-t)\dr \,\beta(c_n,Y).
		\end{equation}
		Therefore since $t<1$ in $B_1\setminus B_{1/2}$, from \eqref{eq:epi_beta_bdd} we conclude the proof of the step.
		
		\noindent \underline{Proof of \textbf{Step 2}.} We are going to prove that 
		\begin{equation}\label{eq:epi4}
			-\Delta(w_i-w_j)=0\quad\text{in }\Omega_{ij},
		\end{equation}
		for all $i,j=1,2,3$, $i\neq j$, where
		\begin{equation*}
			\Omega_{ij}:=\left(\Omega_i\cup\Omega_j\cup\left(\partial\Omega_i\cap\partial\Omega_j\right)\right)\cap B_1.
		\end{equation*}
		For the sake of simplicity, we prove it for $i=1$ and $j=2$. First of all, we fix any $x_0\in\Omega_{12}$ and $r>0$ in such a way that $\overline{B_{2r}(x_0)}\sub \Omega_{12}$. Moreover, for fixed $\tau>0$, we let
		\begin{equation*}
			C_\tau:=\left\{(x'',x_{d-1},x_d)\in\R^d\colon \sqrt{x_{d-1}^2+x_d^2}<\tau\right\}
		\end{equation*}
		and we assume that $B_{2r}(x_0)\cap C_\tau=\emptyset$ and that $c_{n,j}=0$ in $B_{2r}(x_0)$, for $j=3,\dots,N_n$. This is possible by choosing $\tau$ sufficiently small and $n$ sufficiently large, in view of the Hausdorff convergence assumption \eqref{eq:epi5}. We now use \eqref{eq:epi2} with $u\in H^1(B_1;\Sigma_{N_n})$ being such that $u=c_n$ in $B_1\setminus B_{2r}(x_0)$ and $u_j=0$ in $B_{2r}(x_0)$ for $j=3,\dots,N_n$. Since $c_{n,j}=u_{n,j}\equiv 0$ in $B_{2r}(x_0)$, we have that
		$$\beta(u,Y)=\beta(c_n,Y),$$
		which, in view of Step 1, gives
		\begin{equation}\label{eq:beta-mid-step}
			\frac1{\delta_n^2}\beta(u,Y)-\frac{1-\eps_n}{\delta_n^2}\beta(c_n,Y)=\frac{\eps_n}{\delta_n^2}\beta(c_n,Y)=o(1),
		\end{equation}
		as $n\to+\infty$. Moreover, we observe that
		\begin{equation*}
			\frac{\d_{\Sigma_{N_n}}(u_n,Y)}{\delta_n}=\frac{1}{\delta_n}\Bigg|Y_i-c_{n,i}+\sum_{j\neq i} c_{n,j}\Bigg|=w_{n,i}\quad\text{in }\Omega_i\setminus B_{2r}(x_0),
		\end{equation*}
		for $i=1,2,3$. Now, from \eqref{eq:epi2} and \eqref{eq:beta-mid-step} we deduce that
		\begin{multline*}
			\int_{\Omega_1\cap B_{2r}(x_0)}|\nabla w_{n,1}|^2\dx +\int_{\Omega_2\cap B_{2r}(x_0)}|\nabla w_{n,2}|^2\dx +o(1)\\
			<\int_{\Omega_1\cap B_{2r}(x_0)}\Bigg|\nabla\left(\frac{Y_1-u_1+u_2}{\delta_n}\right)\Bigg|^2\dx+\int_{\Omega_2\cap B_{2r}(x_0)}\Bigg|\nabla\left(\frac{Y_2-u_2+u_1}{\delta_n}\right)\Bigg|^2\dx,
		\end{multline*}
		for all $u\in H^1(B_1;\Sigma_2)$ such that $u=c_n$ on $\partial B_{2r}(x_0)$, where the reminder term does not depend on $u$. In particular, if we let
		\begin{equation*}
			W_n:=w_{n,1}-w_{n,2}\quad\text{and}\quad U:=u_1-u_2
		\end{equation*}
		this implies that
		\begin{equation}\label{eq:epi6}
			\int_{B_{2r}(x_0)}|\nabla W_n|^2\dx<\int_{B_{2r}(x_0)}\Bigg|\nabla\Bigg(\frac{Y_1-Y_2-U}{\delta_n}\Bigg)\Bigg|^2\dx+o(1)
		\end{equation}
		for all $U\in H^1(B_{2r}(x_0))$ such that $U=c_{n,1}-c_{n,2}$ on $\partial B_{2r}(x_0)$. One can also see that, by definition,
		\begin{equation*}
			w_{n,1}=-w_{n,2}\quad\text{and}\quad \nabla w_{n,1}=-\nabla w_{n,2}\quad\text{on }\left(\partial\Omega_1\cap\partial\Omega_2\right)\cap B_{2r}(x_0),
		\end{equation*}
		so that $W_n\in H^1(B_{2r}(x_0))\cap C^{0,1}(B_{2r}(x_0))$; moreover, its $H^1$ norm is bounded and so, up to a subsequence,
		\begin{equation*}
			W_n\weak w_1-w_2\quad\text{weakly in }H^1(B_{2r}(x_0)),~\text{as }n\to\infty.
		\end{equation*}
		At this point, we define
		\begin{equation*}
			U=\eta(Y_1-Y_2-\delta_n\phi)+(1-\eta)(c_{n,1}-c_{n,2}),
		\end{equation*}
		where $\eta\in C_c^\infty(B_{2r}(x_0))$ is such that $\eta=1$ in $B_r(x_0)$ and $\phi\in H^1(B_{2r}(x_0))$. Plugging this choice into \eqref{eq:epi6}, we get
		\begin{equation*}
			\int_{B_{2r}(x_0)}|\nabla W_n|^2\dx<\int_{B_{2r}(x_0)}|\nabla((1-\eta)W_n+\eta\phi)|^2\dx+o(1).
		\end{equation*}
		By direct computations, one can see that
		\begin{multline*}
			\int_{B_{2r}(x_0)}(1-(1-\eta)^2)|\nabla W_n|^2\dx+o(1) \\<\int_{B_{2r}(x_0)}\left(W_n^2|\nabla\eta|^2-2(1-\eta)W_n\nabla W_n\cdot\nabla \eta+|\nabla(\eta\phi)|^2+2\nabla((1-\eta)W_n)\cdot\nabla(\eta\phi)\right)\dx
		\end{multline*}
		and, since we can pass now to the limit as $n\to\infty$ in the inequality above, after rearranging the terms back we obtain that
		\begin{equation*}
			\int_{B_{2r}(x_0)}|\nabla (w_1-w_2)|^2\dx\leq \int_{B_{2r}(x_0)}|\nabla ((1-\eta)(w_1-w_2)+\eta\phi)|^2\dx.
		\end{equation*}
		Finally, we take $\phi=w_1-w_2$ in $B_{2r}(x_0)\setminus B_{r}(x_0)$, so that
		\begin{equation*}
			\int_{B_r(x_0)}|\nabla (w_1-w_2)|^2\dx\leq \int_{B_r(x_0)}|\nabla\phi|^2\dx.
		\end{equation*}
		Hence, we proved that $\Delta(w_1-w_2)=0$ in $B_r(x_0)$. Up to moving the ball, we get \eqref{eq:epi4}.\medskip
		
		\noindent \underline{Proof of \textbf{Step 3}.} We use the classification of the $\sfrac32$-homogeneous solutions of the linearized problem, namely \Cref{prop:classification}. More precisely, if we denote
		\begin{equation*}
			\widetilde{w}:=-w_1+w_2-w_3,
		\end{equation*}
		which, we recall, belongs to $H^1(B_1\setminus\mathcal{P})$, then in view of \Cref{prop:classification} we have that
		\begin{equation*}
			\widetilde{w}=w_{\tu{e}}+w_{\tu{o}},
		\end{equation*}
		where $w_{\tu{e}}$ is even with respect to $x_d$ and is of the form
		\begin{equation*}
			w_{\tu{e}}(x)=a_0\widetilde{Y}(x_{d-1},x_d)+U_0(x_{d-1},x_d)\sum_{j=1}^{d-2}a_j x_j,\quad\text{for }x_d>0,
		\end{equation*}
		for some $a_j\in\R$, $j=0,\dots,d-2$, and $w_{\tu{o}}$ is odd with respect to $x_d$ and is of the form
		\begin{equation*}
			w_{\tu{o}}(x)=b_0\widetilde{Y}(-x_{d-1},x_d)+U_0(-x_{d-1},x_d)\sum_{j=1}^{d-2}b_j x_j\quad\text{for }x_d>0,
		\end{equation*}
		for some $b_j\in\R$, $j=0,\dots,d-2$, where
		\begin{equation}\label{eq:Y_tilde}
			\widetilde{Y}(r,\theta)=r^{\frac{3}{2}}\cos\left(\frac{3}{2}\theta\right)=-Y_1+Y_2-Y_3\quad\text{and}\quad 	U_0(r,\theta)=r^{\frac{1}{2}}\cos\left(\frac{\theta}{2}\right).
		\end{equation}
		Moreover, if we denote by $Z$ and $V_0$, respectively, the odd extensions of $\widetilde{Y}(-x_{d-1},x_d)$ and $U_0(-x_{d-1},x_d)$ and if we pass to polar coordinates in the last two variables, i.e. 
		\begin{equation*}
			\begin{bvp}
				x_{d-1}&=r\cos\theta \\
				x_d&=r\sin\theta,
			\end{bvp}\qquad\text{with }r\in(0,1)~\text{and}~ \theta\in\left(-\pi,\pi\right),
		\end{equation*}
		then we have
		\begin{equation}\label{eq:epi9}
			Z(r,\theta)=-r^{\frac{3}{2}}\sin\left(\frac{3}{2}\theta\right)\quad\text{and}\quad V_0(r,\theta)=r^{\frac{1}{2}}\sin\left(\frac{\theta}{2}\right).
		\end{equation}
		We also observe that, by explicit computations
		\begin{equation}\label{eq:epi8}
			\partial_{x_{d-1}} \widetilde{Y}(r,\theta)=\frac{3}{2}r^{\frac{1}{2}}\cos\left(\frac{\theta}{2}\right)\quad\text{and}\quad \partial_{x_d} \widetilde{Y}(r,\theta)=-\frac{3}{2}r^{\frac{1}{2}}\sin\left(\frac{\theta}{2}\right).
		\end{equation}
		We now prove that all the coefficients $a_j,b_j$, with $j=0,\dots, d-2$ vanish, by exploiting the fact that $Y$ is the projection of $c_n$ onto the space of blow-up limits $\mathcal{M}_{\frac{3}{2}}(\R^d;3)$. We know that
		\begin{equation}\label{eq:proj_L2_epi}
			\norm{\d_{\Sigma_{N_n}}(c_n,Y)}_{L^2(B_1)}^2\leq \norm{\d_{\Sigma_{N_n}}(c_n,h)}_{L^2(B_1)}^2\quad\text{for all }h\in \mathcal{M}_{\frac{3}{2}}(\R^d;3)
		\end{equation}
		and from this we can deduce that
            \begin{equation}\label{eq:proj_1}
            \begin{aligned}
                &2\sum_{i=1}^{N_n}\int_{\{c_{n,i}>0\}}(c_{n,i}-\widehat{Y}_i)(\widehat{h}_i-\widehat{Y}_i)\dx -2\int_{\{c_n=0\}}\widehat{Y}_1(\widehat{h}_1-\widehat{Y}_1)\dx\\
                &\leq \sum_{i=1}^{N_n}\int_{\{c_{n,i}>0\}}(\widehat{h}_i-\widehat{Y}_i)^2\dx+\int_{\{c_n=0\}}(\widehat{h}_1-\widehat{Y}_1)^2\dx,
            \end{aligned}
            \end{equation}
            where, for the sake of brevity, we assumed that every integral is always taken inside $B_1$, avoiding to write $\cap B_1$. Moreover, we recall that the $\widehat{~}$ operation is defined in \eqref{eq:hat_op}. Let us prove \eqref{eq:proj_1}. First of all, we observe that, since
            \begin{equation*}
                \d_{\Sigma_{N_n}}(c_n,Y)^2=|c_{n,i}-\widehat{Y}_i|^2\quad\text{on }\overline{\{c_{n,i}>0\}},
            \end{equation*}
            and since
            \begin{equation*}
                \d_{\Sigma_{N_n}}(c_n,Y)^2=|\widehat{Y}_j|^2=|Y|^2\quad\text{in }\{c_n=0\},\quad\text{for any }j=1,2,3,
            \end{equation*}
            we have that
            \begin{align*}
                \norm{\d_{\Sigma_{N_n}}(c_n,Y)}_{L^2(B_1)}^2&=\int_{\{c_n=0\}}|\widehat{Y}_1|^2\dx+\sum_{i=1}^{N_n}\int_{\{c_{n,i}>0\}}(c_{n,i}-\widehat{Y}_i)^2\dx.
            \end{align*}
            Hence, reasoning analogously for the right-hand side of \eqref{eq:proj_L2_epi}, we get that
            \begin{equation*}
                \int_{\{c_n=0\}}|\widehat{Y}_1|^2\dx+\sum_{i=1}^{N_n}\int_{\{c_{n,i}>0\}}(c_{n,i}-\widehat{Y}_i)^2\dx\leq \int_{\{c_n=0\}}|\widehat{h}_1|^2\dx+\sum_{i=1}^{N_n}\int_{\{c_{n,i}>0\}}(c_{n,i}-\widehat{h}_i)^2\dx.
            \end{equation*}
            Then, by explicit calculations we get \eqref{eq:proj_1}. 
            
            We now fix $h\in\mathcal{M}_{\frac{3}{2}}(\R^d;3)$ and claim that
            \begin{multline}\label{eq:proj_2}
                2\sum_{i=1}^3\int_{\Omega_i}(Y_i-\widehat{c}_{n,i})(Y_i-\widehat{h}_i)\dx \\ -4\sum_{i=1}^3\sum_{j\neq i}\int_{\Omega_i\cap\{c_{n,j}>0\}}(Y_i-\widehat{c}_{n,i})\Big(\sum_{k\neq i,j}h_k\Big)\dx\leq \sum_{i=1}^3\int_{\Omega_i}(Y_i-\widehat{h}_i)^2\dx.
            \end{multline}
            We start by taking into account the first term in the left-hand side of \eqref{eq:proj_1}. We split and rearrange the integrals as follows
            \begin{align*}
                &\sum_{i=1}^{N_n}\int_{\{c_{n,i}>0\}}(c_{n,i}-\widehat{Y}_i)(\widehat{h}_i-\widehat{Y}_i)\dx =\sum_{i=1}^{N_n}\sum_{j=1}^3\int_{\{c_{n,i}>0\}\cap \Omega_j}(c_{n,i}-\widehat{Y}_i)(\widehat{h}_i-\widehat{Y}_i)\dx =\\
                &=\sum_{j=1}^3\Bigg(\int_{\{c_{n,j}>0\}\cap \Omega_j}(c_{n,j}-\widehat{Y}_j)(\widehat{h}_j-\widehat{Y}_j)\dx +\sum_{i\neq j}\int_{\{c_{n,i}>0\}\cap\Omega_j}(c_{n,i}-\widehat{Y}_i)(\widehat{h}_i-\widehat{Y}_i)\dx\Bigg) \\
                &=\sum_{j=1}^3\Bigg(\int_{\{c_{n,j}>0\}\cap \Omega_j}(\widehat{c}_{n,j}-Y_j)(\widehat{h}_j-Y_j)\dx +\sum_{i\neq j}\int_{\{c_{n,i}>0\}\cap\Omega_j}(-\widehat{c}_{n,j}+Y_j)(\widehat{h}_i+Y_j)\dx\Bigg) \\
                &=\sum_{j=1}^3\Bigg(\int_{\{c_{n,j}>0\}\cap \Omega_j}(\widehat{c}_{n,j}-Y_j)(\widehat{h}_j-Y_j)\dx +\sum_{i\neq j}\int_{\{c_{n,i}>0\}\cap\Omega_j}(\widehat{c}_{n,j}-Y_j)(-\widehat{h}_i-Y_j)\dx\Bigg).
            \end{align*}
            We now observe that, since for $i\neq j$,
            \begin{equation}\label{eq:proj_4}
                \widehat{h}_j+\widehat{h}_i=-2\sum_{k\neq i,j}h_k\,,
            \end{equation}
            we have
            \begin{equation*}
                -\widehat{h}_i-Y_j=\widehat{h}_j-Y_j-(\widehat{h}_j+\widehat{h}_i)=\widehat{h}_j-Y_j+2\sum_{k\neq i,j}h_k.
            \end{equation*}
            Now, plugging this identity into the previous one, we get
            \begin{align}
                \sum_{i=1}^{N_n}\int_{\{c_{n,i}>0\}}(c_{n,i}-\widehat{Y}_i)(\widehat{h}_i-\widehat{Y}_i)\dx&=\sum_{j=1}^3\Bigg(\sum_{i=1}^{N_n}\int_{\{c_{n,i}>0\}\cap \Omega_j}(\widehat{c}_{n,j}-Y_j)(\widehat{h}_j-Y_j)\dx\notag \\
                &\quad+2\sum_{i\neq j}\int_{\{c_{n,i}>0\}\cap \Omega_j}(\widehat{c}_{n,j}-Y_j)\Big(\sum_{k\neq i,j}h_k\Big)\dx\Bigg).\label{eq:proj_3}
            \end{align}
            Let us now consider the second term in the left-hand side of \eqref{eq:proj_1}. Since 
            $$|\widehat Y_1|^2=Y_1^2+Y_2^2+Y_3^2,$$
            and
            \begin{align*}
            \widehat Y_1\widehat h_1&=(Y_1-Y_2-Y_3)(h_1-h_2-h_3)\\
            &=Y_1(h_1-h_2-h_3)+Y_2(h_2-h_1-h_3+2h_3)+Y_3(h_3-h_1-h_2+2h_2)\\
            &=Y_1\widehat h_1+Y_2\widehat h_2+Y_3\widehat h_3+2(Y_2h_3+Y_3h_2),
            \end{align*}
            we have that 
            \begin{align*}
                -\int_{\{c_n=0\}}\widehat{Y}_1(\widehat{h}_1-\widehat{Y}_1)\dx&=-\sum_{j=1}^3\int_{\{c_n=0\}\cap \Omega_j}\widehat{Y}_1(\widehat{h}_1-\widehat{Y}_1)\dx  \\
                &= -\sum_{j=1}^3\int_{\{c_n=0\}\cap \Omega_j}Y_j(\widehat{h}_j-Y_j)\dx-2\int_{\{c_n=0\}}(Y_2h_3+Y_3h_2)\dx\\
                &=\sum_{j=1}^3\int_{\{c_n=0\}\cap \Omega_j}(\widehat{c}_{n,j}-Y_j)(\widehat{h}_j-Y_j)\dx -2\int_{\{c_n=0\}}(Y_2h_3+Y_3h_2)\dx.
            \end{align*}
            Summing the previous identity with \eqref{eq:proj_3}, we obtain that the left-hand side of \eqref{eq:proj_1} can be written as
            \begin{align*}
                &2\sum_{i=1}^{N_n}\int_{\{c_{n,i}>0\}}(c_{n,i}-\widehat{Y}_i)(\widehat{h}_i-\widehat{Y}_i)\dx -2\int_{\{c_n=0\}}\widehat{Y}_1(\widehat{h}_1-\widehat{Y}_1)\dx\\
                &=2\sum_{j=1}^3\Bigg(\sum_{i=1}^{N_n}\int_{\{c_{n,i}>0\}\cap \Omega_j}(\widehat{c}_{n,j}-Y_j)(\widehat{h}_j-Y_j)\dx+2\sum_{i\neq j}\int_{\{c_{n,i}>0\}\cap \Omega_j}(\widehat{c}_{n,j}-Y_j)\Big(\sum_{k\neq i,j}h_k\Big)\dx\Bigg)\\
                &\quad+2\sum_{j=1}^3\int_{\{c_n=0\}\cap \Omega_j}(\widehat{c}_{n,j}-Y_j)(\widehat{h}_j-Y_j)\dx -4\int_{\{c_n=0\}}(Y_2h_3+Y_3h_2)\dx\\
                &=2\sum_{j=1}^3\Bigg(\int_{\Omega_j}(\widehat{c}_{n,j}-Y_j)(\widehat{h}_j-Y_j)\dx +2\sum_{i\neq j}\int_{\{c_{n,i}>0\}\cap \Omega_j}(\widehat{c}_{n,j}-Y_j)\Big(\sum_{k\neq i,j}h_k\Big)\dx\Bigg)\\
                &\quad -4\int_{\{c_n=0\}}(Y_2h_3+Y_3h_2)\dx,
            \end{align*}
            where we notice that in the second-last line we have precisely left-hand side of \eqref{eq:proj_2} (up to exchanging the roles of $i$ and $j$), while the last term will simplify with a corresponding equal term on the right-hand side.

            We now work on the right-hand side of \eqref{eq:proj_1}. By splitting the integral in the first term, we get that
            \begin{align}
                &\sum_{i=1}^{N_n}\int_{\{c_{n,i}>0\}}(\widehat{h}_i-\widehat{Y}_i)^2\dx=\sum_{i=1}^{N_n}\sum_{j=1}^3\int_{\{c_{n,i}>0\}\cap\Omega_j}(\widehat{h}_i-\widehat{Y}_i)^2\dx = \notag\\
                &=\sum_{j=1}^3\Bigg(\int_{\{c_{n,j}>0\}\cap \Omega_j}(\widehat{h}_j-Y_j)^2\dx+\sum_{i\neq j}\int_{\{c_{n,i}>0\}\cap \Omega_j}(\widehat{h}_i+Y_j)^2\dx\Bigg).\label{eq:proj_5}
            \end{align}
            We now notice that, in view of \eqref{eq:proj_4}, for $i\neq j$ we have
            \begin{align*}
                (\widehat{h}_i+Y_j)^2=(\widehat{h}_i+\widehat{h}_j+Y_j-\widehat{h}_j)^2&=\Big((Y_j-\widehat{h}_j)^2+4\Big(\sum_{k\neq i,j} h_k\Big)\Big(\widehat{h}_j-Y_j+\sum_{k\neq i,j}h_k\Big)\Big) \\
                &=\Big((Y_j-\widehat{h}_j)^2+4\Big(\sum_{k\neq i,j} h_k\Big)\Big(h_j-h_i-Y_j\Big)\Big) \\
                &=\Big((Y_j-\widehat{h}_j)^2-4\Big(\sum_{k\neq i,j} h_k\Big)Y_j\Big)\leq (Y_j-\widehat{h}_j)^2,
            \end{align*}
            where in the second last step we used the fact that $h=(h_1,h_2,h_3)$ is segregated and in the last step that $h_k,Y_j\geq 0$. Hence, together with \eqref{eq:proj_5}, this tells us that
            \begin{equation}\label{eq:proj_6}
                \sum_{i=1}^{N_n}\int_{\{c_{n,i}>0\}}(\widehat{h}_i-\widehat{Y}_i)^2\dx\leq \sum_{j=1}^3\sum_{i=1}^{N_n}\int_{\{c_{n,i}>0\}\cap \Omega_j}(\widehat{h}_j-Y_j)^2\dx
            \end{equation}
            Finally, concerning the second term in the right-hand side of \eqref{eq:proj_1}, we have
            \begin{align}
                \int_{\{c_n=0\}}(\widehat{h}_1-\widehat{Y}_1)^2\dx&=\sum_{j=1}^3\int_{\{c_n=0\}\cap\Omega_j}\big((\widehat{h}_j-\widehat{Y}_j)^2+2(\widehat{h}_j\widehat{Y}_j-\widehat{h}_1\widehat{Y}_1)\big)\dx \notag \\
                &=\sum_{j=1}^3\int_{\{c_n=0\}\cap\Omega_j}(\widehat{h}_j-Y_j)^2\dx-4\int_{\{c_n=0\}}(Y_2h_3+Y_3h_2)\dx,\label{eq:proj_7}
            \end{align}
            where, in the last step, we used that 
            \begin{align*}
                &\widehat{h}_2\widehat{Y}_2-\widehat{h}_1\widehat{Y}_1=Y_2(\widehat{h}_1+\widehat{h}_2)=-2Y_2h_3\quad\text{in }\Omega_2, \\
                &\widehat{h}_3\widehat{Y}_3-\widehat{h}_1\widehat{Y}_1=Y_3(\widehat{h}_1+\widehat{h}_3)=-2Y_3h_2\quad\text{in }\Omega_3.
            \end{align*}
            Therefore, summing up \eqref{eq:proj_6} and \eqref{eq:proj_7}, we get that the right-hand side of \eqref{eq:proj_1} is smaller than or equal to the right-hand side of \eqref{eq:proj_2}, recalling that the term
            \begin{equation*}
                -4\int_{\{c_n=0\}}(Y_2h_3+Y_3h_2)\dx
            \end{equation*}
            simplifies with a corresponding one in the left-hand side. Hence \eqref{eq:proj_2} is proved. 
            We now rearrange the terms in \eqref{eq:proj_2} and divide by $2\delta_n^3$, in order to obtain that 
            \begin{equation}\label{eq:epi3}
                \sum_{i=1}^3\int_{\Omega_i}w_{n,i}\,\frac{Y_i-\widehat{h}_i}{\delta_n^2}\dx \leq  \frac{\delta_n}{2}\sum_{i=1}^3\int_{\Omega_i}\Bigg|\frac{Y_i-\widehat{h}_i}{\delta_n^2}\Bigg|^2\dx+R_n(h)\quad\text{for any }h\in\mathcal{M}_{\frac{3}{2}}(\R^d;3),
            \end{equation}
            where
            \begin{equation}\label{eq:R_n}
                R_n(h):=2\sum_{i=1}^3\sum_{j\neq i}\int_{\Omega_i\cap\{c_{n,j}>0\}}w_{n,i}\,\Big(\sum_{k\neq i,j}\frac{h_k}{\delta_n^2}\Big)\dx.
            \end{equation}
		At this point, we consider as $h$ suitable perturbations of the blow-up $Y$ (depending on the parameter $\delta_n$) and pass to the limit as $n\to\infty$. This will provide a set of orthogonality conditions which the limit $w$ must satisfy; in view of the classification, this leads to conditions on the coefficients $a_j$ and $b_j$. More precisely, we are going to choose a multiple of the blow-up limit $Y$ and the infinitesimal rotations inside the coordinate planes $x_ix_j$, with $i=1,\dots,d-1$ and $j=d-1,d$ ($i\neq j$). This way, we get a total of $2(d-1)$ equations which coincide with the degrees of freedom of the problem. We proceed in the following steps.
		\begin{enumerate}
			\item $a_0=0$: we choose $h_{n}(x)=(1\pm{\delta_n^2})Y(x_{d-1},x_d)$.
			\item $a_j=0$ for $j=1,\dots,d-2$: we choose $h_{n}(x)=Y(x_{d-1}\cos({\delta_n^2})\pm x_j\sin({\delta_n^2}),x_d)$.
			\item $b_0=0$: we choose $h_{n}(x)=Y( x_{d-1}\cos({\delta_n^2})\mp x_d\sin({\delta_n^2}), x_d\cos({\delta_n^2})\pm x_{d-1}\sin({\delta_n^2}))$.
			\item $b_j=0$ for $j=1,\dots,d-2$: we choose $h_{n}(x)=Y(x_{d-1},x_d\cos({\delta_n^2})\pm x_j\sin({\delta_n^2}))$.
		\end{enumerate}
        We notice that, if $h_n$ is as in (1)--(4), by Lipschitz continuity there holds
        \begin{equation*}
            0\leq h_{n,k}\leq C\delta_n^2\quad\text{on }\Omega_i\cap \{h_{n,k}>0\},
        \end{equation*}
        for some $C>0$ independent from $n$, for any $i\neq k$. As a consequence
        \begin{multline}\label{eq:R_n_0}
            |R_n(h_n)|\leq 2C\sum_{i=1}^3\sum_{k\neq i}\int_{\Omega_i\cap \{h_{n,k}>0\}}|w_{n,i}|\dx \\ 
            \leq 2C\sum_{i=1}^3\norm{w_{n,i}}_{L^2(\Omega_i)}\sum_{k\neq i}|\Omega_i\cap\{h_{n,k}>0\}|\to 0
        \end{multline}
        as $n\to\infty$, where we also used Cauchy-Schwarz inequality, the fact that $w_{n,i}$ is bounded in $L^2(\Omega_i)$ with respect to $n$ and that $|\Omega_i\cap\{h_{n,k}>0\}|=|\{Y_i>0\}\cap\{h_{n,k}>0\}|\to 0$ as $n\to\infty$ for $k\neq i$. This last fact holds since $h_n$ is the composition of $Y$ and a rotation at angle $\delta_n^2$ in cases (2)--(4) and the product of $Y$ by a constant in case (1). Hence, when using \eqref{eq:epi3}, we write $o(1)$ in place of $R_n(h_n)$.
		We also observe that perturbations (1) and (2) leave the last coordinate $x_d$ untouched, and so they generate conditions only on the even part of $w$. Let us now perform the various computations.
		\begin{enumerate}
			\item Plugging $h$ into \eqref{eq:epi3}, we get that
			\begin{equation*}
				\mp\sum_{i=1}^3\left(w_{n,i},Y_i\right)_{L^2(\Omega_i)}\leq \frac{\delta_n}{2}\sum_{i=1}^3\norm{Y_i}_{L^2(\Omega_i)}^2{+o(1)},
			\end{equation*}
			and, passing to the limit as $n\to\infty$, this leads to 
			\begin{equation*}
				\sum_{i=1}^3(\widetilde{w},\widetilde{Y})_{L^2(\Omega_i)}=(\widetilde{w},\widetilde{Y})_{L^2(B_1)}=0,
			\end{equation*}
			being $\widetilde{w},\widetilde{Y}\in H^1(B_1\setminus\mathcal{P})$, where we recall that $\widetilde Y=\widehat Y_2$. In view of the classification of the solutions of the linearized problem, since $\widetilde{Y}$ is even with respect to $x_d$ and independent from $x_j$, for $j=1,\dots,d-2$, the terms in $\widetilde{w}$ which linearly depend on $x_j$ do not contribute; thus, we have
			\begin{equation*}
				a_0\|\widetilde{Y}\|_{L^2(B_1)}^2=(\widetilde{w},\widetilde{Y})_{L^2(B_1)}=0,
			\end{equation*}
			which implies that $a_0=0$.
			\item First of all, we observe that
            \begin{equation*}
                \frac{Y_i-\widehat{h}_{n,i}}{\delta_n^2}\to\mp \frac{3}{2}\,x_j\,U_0
            \end{equation*}
            strongly in ${L^2}(\Omega_i)$, as $n\to\infty$, for all $i=1,2,3$ and $j=1,\dots,d-2$. Therefore, plugging $h$ into \eqref{eq:epi3}, using \eqref{eq:R_n_0} and passing to the limit as $n\to\infty$ yields
			\begin{equation*}
				\sum_{i=1}^3(\widetilde{w},x_jU_0)_{L^2(\Omega_i)}=(\widetilde{w},x_j U_0)_{L^2(B_1)}=0,
			\end{equation*}
            where $U_0$ is defined in \eqref{eq:Y_tilde}.
			On the other hand, since $\widetilde{Y}$ is even with respect to $x_d$ and since $x_j U_0$ is independent from the variable $x_k$ for $k=1,\dots, d-2$, $k\neq j$, thanks to the explicit form of $\widetilde{w}$ we infer
			\begin{equation*}
				(\widetilde{w},x_j U_0)_{L^2(B_1)}=a_j\norm{x_j U_0}_{L^2(B_1)}^2=0,
			\end{equation*}
			thus proving that $a_j=0$.
			\item By \eqref{eq:epi9} and \eqref{eq:epi8} we deduce that
                \begin{equation*}
                    \frac{Y_i-\widehat{h}_{n,i}}{\delta_n^2}\to \mp Z
                \end{equation*}
			 strongly in $L^2(\Omega_i)$, for $i=1,2,3$, as $n\to\infty$, where $Z$ is defined in \eqref{eq:epi9}.
			Hence, reasoning analogously to the previous step, from \eqref{eq:epi3}, using \eqref{eq:R_n_0} we get
			\begin{equation*}
				\left(\widetilde{w},Z\right)_{L^2(B_1)}=0.
			\end{equation*}
			Since the second term of the scalar product is odd with respect to $x_d$ and since the right term does not depend on $x_j$ for any $j=1,\dots,d-2$, in view also of the explicit form of $\widetilde{w}$, this implies that
			\begin{equation*}
	b_0\norm{Z}_{L^2(B_1)}^2=0.
			\end{equation*}
			thus implying that $b_0=0$.
			
			\item Again, in view of \eqref{eq:epi9} and \eqref{eq:epi8} we have that
                \begin{equation*}
				\frac{Y_i-\widehat{h}_{n,i}}{\delta_n^2}\to \pm\frac{3}{2}x_j V_0(x_{d-1},x_d)
			\end{equation*}
			strongly in $L^2(\Omega_i)$, for $i=1,2,3$, as $n\to\infty$, where $V_0$ is defined in \eqref{eq:epi9}. Passing to the limit in \eqref{eq:epi3}, reasoning as in step (2) and (3), and using \eqref{eq:R_n_0} we get
			\begin{equation*}
				b_j\norm{x_jV_0}_{L^2(B_1)}^2=0.
			\end{equation*}
			which forces $b_j=0$, thus completing the proof of the claim.
		\end{enumerate}

		\noindent \underline{Proof of \textbf{Step 4}.} For any $n\in\N$, we define $u_n\in H^1(B_1;\Sigma_{N_n})$ and $\bar{u}_{n,i}$, $i=1,2,3$, as in Step 1. Now, we compute the energy of $w_{n,i}$ and $\bar{u}_{n,i}$ on each $\Omega_i$ (setting $\Omega_i':=\partial\Omega_i\cap \partial B_1$, $\gamma:=3/2$ and extending $t=0$ for $r<1/2$), obtaining that
		\begin{equation*}
			\int_{\Omega_i}|\nabla w_{n,i}|^2\dx=\int_0^1 r^{d-1}|\partial_r(r^\gamma)|^2\dr\int_{\Omega_i'}w_{n,i}^2\ds+\int_0^1 r^{d+2\gamma-3}\dr\int_{\Omega_i'}|\nabla_{\partial B_1} w_{n,i}|^2\ds
		\end{equation*}
		and that
		\begin{align*}
			\int_{\Omega_i}|\nabla \bar{u}_{n,i}|^2\dx&=\int_{\Omega_i\setminus B_{\frac{1}{2}}}|\nabla (tw_{n,i})|^2\dx \\
			&=\int_0^1 r^{d-1}|\partial_r(tr^\gamma)|^2\dr\int_{\Omega_i'}w_{n,i}^2\ds+\int_0^1 r^{d+2\gamma-3}t^2\dr\int_{\Omega_i'}|\nabla_{\partial B_1} w_{n,i}|^2\ds.
		\end{align*}
		Plugging these computations into \eqref{eq:epi2} and recalling \eqref{eq:epi_beta_diff}, we get that
		\begin{multline*}
			\int_0^1r^{d+2\gamma-3}(1-t^2)\dr\sum_{i=1}^3\int_{\Omega_i'}|\nabla_{\partial B_1} w_{n,i}|^2\ds \\
			<\int_0^1r^{d-1}\left(|\partial_r(tr^\gamma)|^2-|\partial_r(r^\gamma)|^2\right)\dr\sum_{i=1}^3 \int_{\Omega_i'}w_{n,i}^2\ds\\
			\qquad- (d+1)\int_0^1r^d(1-\eps_n-t)\dr \,\frac{\beta(c_n,Y)}{\delta_n^2}+o(1),
		\end{multline*}
		as $n\to\infty$. Since
		\begin{equation*}
			\int_0^1r^{d+2\gamma-3}(1-t^2)\dr>0,\quad (d+1)\int_0^1r^d(1-\eps_n-t)\dr \,\frac{\beta(c_n,Y)}{\delta_n^2}\geq 0
		\end{equation*}
		for $n$ sufficiently large and since, by compact trace embedding,
		\begin{equation*}
			w_{n,i}\to 0\quad\text{strongly in }L^2(\Omega_i'),~\text{as }n\to\infty,
		\end{equation*}
		we deduce that $w_{n,i}\to 0$ strongly in $H^1(\Omega_i)$, for all $i=1,2,3$, thus reaching a contradiction and concluding the proof.
	\end{proof}
	
	\section{\texorpdfstring{Regularity of the free boundary around points of frequency $\sfrac32$}{Regularity of the free boundary around points of frequency 3/2}}\label{sec:regularity}

    In this section we complete the proof of \cref{t:main-main} by proving the regularity of the free interface $\mathcal F(u)$ around points of frequency $\sfrac32$. The main ingredients are the classification of the $\sfrac32$-homogeneous blow-ups (see \Cref{prop:32-homogeneous-classification}) and the epiperimetric inequality \cref{t:epi}.
    
	\subsection{Rate of convergence of the blow-up sequence}\label{sub:uniqueness-of-the-blow-up}
    We will show that we can apply the epiperimetric inequality to the rescalings of $u$ of the form $u^{r,x_0}(x)=r^{-\sfrac32}u(x_0+rx)$. In order to do so, we need to show that the conditions of \cref{t:epi} are fulfilled uniformly, at all points $x_0$ and at all small scales $r$, by multiples of $u^{r,x_0}$. We start with the following.
	\begin{lemma}[Uniform Hausdorff distance estimate]\label{lemma:eps_reg_outside}
		Let $d\ge 2$ and $N\ge 1$ be fixed. For any $\eps>0$ there exists $\delta>0$ (depending on $d$, $N$ and $\eps$) such that, if $u\in H^1(B_1;\Sigma_N)\cap C^{0,1}(B_1;\R^N)$ is a minimizer, $u\in \mathcal M(B_1;N)$, that satisfies
		\begin{equation*}
			\norm{\d_{\Sigma_N}(u,Y)}_{L^\infty(B_1)}\leq \delta,
		\end{equation*}
		then
		\begin{equation}\label{e:hausdorff-1}
			\sum_{i=1}^3\d_{\mathcal{H}}\Big(\{u_i>0\}\cap B_{\sfrac12},\{Y_i>0\}\cap B_{\sfrac12}\Big)\leq\eps,
            		\end{equation}
and
            		\begin{equation}\label{e:hausdorff-2}
            \sum_{i=4}^N\d_{\mathcal{H}}\Big(\{u_i>0\}\cap B_{\sfrac12},\{x_{d-1}=x_d=0\}\cap B_{\sfrac12}\Big)\leq \eps.
		\end{equation}
		Moreover, in $B_1\cap \left\{x\in\R^d\colon \sqrt{x_{d-1}^2+x_d^2}>\eps\right\}$ the free boundary $\mathcal{F}(u)$ is the disjoint union of the three surfaces $\Gamma_{ij}=\partial\{u_i>0\}\cap\partial\{u_j>0\}$, for $i,j\in\{1,2,3\}$, $i\neq j$, which are $C^{1,\alpha}$ graphs over $\partial \{Y_i>0\}\cap\partial\{Y_j>0\}$.
	\end{lemma}
	\begin{proof}
		Suppose by contradiction that there are $\eps>0$ and a sequence $u_n\in H^1(B_1;\Sigma_N)\cap C^{0,1}(B_1;\R^N)\cap \mathcal M(B_1;N)$ such that 
		$$\norm{\d_{\Sigma_N}(u_n,Y)}_{L^\infty(B_1)}=\delta_n\to0,$$
		and for which \eqref{e:hausdorff-1} or \eqref{e:hausdorff-2} fail. For any $\tau>0$, we denote
		\begin{align*}
			\mathcal{O}_\tau&:=\{x\in\R^d\colon \dist (x,\{Y=0\})<\tau\}, \\
			C_\tau&:=\left\{x\in\R^d\colon \sqrt{x_{d-1}^2+x_d^2}<\tau\right\}.
		\end{align*}
		Notice there exists $t=t_\eps>0$ such that $\d_{\Sigma_N}(Y,0)>t$ on the set $B_1\setminus\mathcal O_{\eps/6}.$ On the other hand, by the uniform convergence of $u_n$ to $Y$ we get that for $n$ large enough 
		$$\{Y_i>0\}\cap B_1\setminus\mathcal O_{\eps/6}\subset\{Y_i>t\}\cap B_1\subset \{u_{n,i}>0\}\cap B_1\quad\text{for}\quad i=1,2,3.$$ 
		In particular, this excludes the possibility that \eqref{e:hausdorff-1} fails for $u_n$ as $n\to+\infty$. In order to show that \eqref{e:hausdorff-2} cannot fail for large $n$, we notice that for every ball $B_{\eps/3}(x)$ with center $x\in B_{\sfrac12}\cap \partial\{Y_i>0\}\cap\partial\{Y_j>0\}\setminus C_{\eps/2}$ (for $i,j\in\{1,2,3\}$ with $i\neq j$), we have that $Y=Y_i\vec{\bm{e}}_i+Y_j\vec{\bm{e}}_j$ and
		\begin{equation*}
			\norm{\d_{\Sigma_N}(u_n,Y)}_{L^\infty(B_{\eps/3}(x))}\to 0\quad \text{as }n\to\infty.
		\end{equation*}
		Thus, by the epsilon-regularity lemma \cite[Lemma 9.8]{OV}, 
		we get that $u_k\equiv 0$ in $B_{\eps/6}(x)$ for all $k\neq i,j$ and $\mathcal F(u)\cap B_{\eps/6}(x)=\Gamma_{ij}\cap B_{\eps/6}(x)$ is a $C^{1,\alpha}$ graph (with small $C^1$ norm) over the hyperplane $\partial\{Y_i>0\}\cap\partial\{Y_j>0\}$. In particular, this implies that \eqref{e:hausdorff-2} cannot fail and proves that second part of the claim.
	\end{proof}

	\begin{lemma}[Uniform convergence of rescalings with variable center]\label{l:uniform-rescalings}
		Let $u\in \mathcal{M}(B_1;N)$ be such that $\gamma(u,0)=\sfrac{3}{2}$. Then, for all $\eps>0$ there exists $\rho,r_0>0$ (depending on $d$, $N$ and $\eps$) such that for all $z\in B_\rho\cap \mathcal{F}_{\sfrac{3}{2}}(u)$ and for all $r<r_0$ there exists $Y\in\mathcal{M}_{\sfrac{3}{2}}(\R^d;3)$ such that
		\begin{equation*}
			\norm{\d_{\Sigma_N}(u_{z,r},Y)}_{L^\infty(B_{\sfrac12})}\leq \eps,
		\end{equation*}
		where $u_{z,r}(x):=H(u,z,r)^{-\sfrac{1}{2}}u(rx+z)$.
	\end{lemma}
	\begin{proof}
		Let us assume by contradiction that there exists sequences $r_n\to 0$ and $z_n\in\mathcal{F}_{\sfrac{3}{2}}(u)$, with $|z_n|\to 0$, such that
		\begin{equation}\label{eq:compact2}
			\norm{\d_{\Sigma_N}(u_{z_n,r_n},Y)}_{L^\infty(B_{\sfrac12})}>\eps\quad\text{for all }n\in\N
		\end{equation}
		for all $Y\in\mathcal{M}_{\sfrac{3}{2}}(\R^d;3)$, for some $\eps>0$, and that $\gamma(u_{z_n,r_n},0)=\sfrac{3}{2}$ for all $n\in\N$. First of all, we observe that
		\begin{equation}\label{eq:compact1}
			\sum_{i=1}^N\int_{B_1}|\nabla (u_{z_n,r_n})_i|^2\dx=\mathcal{N}(u_{z_n,r_n},0,1).
		\end{equation}
		Thus, up to a subsequence
		\begin{equation}\label{eq:compact3}
			u_{z_n,r_n}\to u_0\quad\text{strongly in }H^1(B_1;\Sigma_N)~\text{and uniformly in }B_{\sfrac12},
		\end{equation}
		for some $u_0\in H^1(B_1;\Sigma_N)$, $u_0\not\equiv 0$. Moreover, since $u_{z_n,r_n}\in \mathcal{M}(B_1;N)$ for all $n$, also $u_0\in\mathcal{M}(B_1;N)$. We now observe that, since $z_n\in\mathcal{F}_{\sfrac{3}{2}}(u)$ and by monotonicity, we have
        \begin{equation*}
            \frac{3}{2}=\mathcal{N}(u,z_n,0^+)\leq \mathcal{N}(u,z_n,Rr_n)=\mathcal{N}(u_{z_n,r_n},0,R),
        \end{equation*}
        for any $R>0$ and $n$ sufficiently large. Passing to the limit as $n\to\infty$, we get
        \begin{equation*}
            \frac{3}{2}\leq \mathcal{N}(u_0,0,R)\quad\text{for all }R>0.
        \end{equation*}
        By the continuity of $x\mapsto \mathcal{N}(u,x,r)$, the monotonicity of $r\mapsto \mathcal{N}(u,x,r)$ and \Cref{l:optimal-gap-main-lemma}, we have that for all $R,\tau>0$ sufficiently small, there holds
		\begin{equation*}
			\frac{3}{2}\leq \mathcal{N}(u_0,0,R)\leq \mathcal{N}(u,0,\tau R).
		\end{equation*}
		Sending $\tau\to 0$ implies that $u_0$ is $\sfrac{3}{2}$-homogeneous, which, in view of \Cref{l:optimal-gap-main-lemma} and \eqref{eq:compact3} contradicts \eqref{eq:compact2}.	
	\end{proof}

We are now in position to prove the uniqueness of the blow-up limit and rate of convergence of the blow-up sequence.

    \begin{proposition}\label{prop:rate}
		Let $\alpha:=\eps(d+1)$, with $\eps>0$ being as in \Cref{t:epi} and let $u\in\mathcal{M}(B_1;N)$. Then, for all compact $K\sub B_1$ there exists $R>0$ such that for all $x_0\in K\cap\mathcal{F}_{\sfrac{3}{2}}(u)$, there exists $C_\tu{rate}>0$ depending on $d$ and $\sup_{x\in K}\mathcal{N}(u,x,\d_{\mathcal{H}}(K,\partial B_1))$ such that
		\begin{equation*}
			\norm{\d_{\Sigma_N}(u^{x_0,r_2},u^{x_0,r_1})}_{L^2(\partial B_1)}^2\leq\sum_{i=1}^N\norm{u_i^{x_0,r_2}-u_i^{x_0,r_1}}_{L^2(\partial B_1)}^2\leq C_\tu{rate}(r_2^\alpha-r_1^\alpha)
		\end{equation*}
		for all $0<r_1<r_2<R$ where $u^{x_0,r}(x):=r^{-\sfrac{3}{2}}u(rx+x_0)$. In particular, there exists $Y^{x_0}\in\mathcal{M}_{\sfrac{3}{2}}(\R^d;3)$ such that
		\begin{equation}\label{eq:rate_L2}
			\norm{\d_{\Sigma_N}(u^{x_0,r},Y^{x_0})}_{L^2(\partial B_1)}^2\leq \sum_{i=1}^N\norm{u_i^{x_0,r}-Y^{x_0}_i}_{L^2(\partial B_1)}^2\leq C_\tu{rate}r^\alpha,
		\end{equation}
		for all $r<R$. Moreover, we have that
        \begin{equation}\label{eq:rate_infty}
            \norm{\d_{\Sigma_N}(u^{x_0,r},Y^{x_0})}_{L^\infty( B_1)}^2\leq C_\tu{rate}'r^{\alpha'},
        \end{equation}
        for some $C_\tu{rate}',\alpha'>0$ depending on $d$ and $\sup_{x\in K}\mathcal{N}(u,x,\d_{\mathcal{H}}(K,\partial B_1))$.
	\end{proposition}
	\begin{proof}
    Thanks to \cref{lemma:eps_reg_outside} and \cref{l:uniform-rescalings}, there is a radius $R$ such that we can apply the epiperimetric inequality to all the rescaling of the form $u^{x_0,r}(x)=r^{-\sfrac32}u(x_0+rx)$ with $r\le R$ and  $x_0\in K$.
    The uniqueness of the blow-up and the rate of convergence now follow by a standard argument (see e.g. \cite[Proposition 8.1]{OV}). Finally, \eqref{eq:rate_infty} follows, after integrating \eqref{eq:rate_L2}, by a standard $L^\infty-L^2$ estimate for uniformly Lipschitz functions (for the proof see e.g. \cite[Lemma 8.4]{OV}) and the Lipschitz estimate for minimizers (see \cite[Theorem 2.4.6]{KS}).
    \end{proof}

As a corollary, we also obtain the following.
	
	\begin{proposition}[Oscillation of the $\sfrac32$-blow-up limits]\label{lemma:osc}
		Let $\bar{\alpha}:=2\alpha/(\alpha+3)$, where $\alpha$ is as in \Cref{prop:rate} and let $u\in\mathcal{M}(B_1;N)$. Then, for all compact $K\sub B_1$ there exists $r_0>0$ (depending on $d$ and $K$) and $C_{\tu{osc}}>0$ (depending on $d$ and $\sup_{x\in K}\mathcal{N}(u,x,\d_{\mathcal{H}}(K,\partial B_1))$) such that
		\begin{equation*}
			\sum_{i=1}^3\norm{Y^{x_0}_i-Y^{z_0}_i}_{L^2(\partial B_1)}^2\leq C_\tu{osc} |x_0-z_0|^{\bar{\alpha}}
		\end{equation*}
		for all $x_0,z_0\in\mathcal{F}_{\sfrac{3}{2}}(u)\cap K$ such that $|x_0-z_0|<r_0$, where
		\begin{equation*}
			Y^p=\lim_{r\to 0}u^{p,r}\quad\text{and}\quad u^{p,r}(x):=r^{-\frac{3}{2}}u(rx+p).
		\end{equation*}
	\end{proposition}
	\begin{proof}
		We first estimate
		\begin{multline}\label{eq:osc_1}
			\sum_{i=1}^3\norm{Y^{x_0}-Y^{z_0}}_{L^2(\partial B_1)}^2\leq 2\sum_{i=1}^N\Big(\norm{u^{x_0,r}_i-Y^{x_0}_i}_{L^2(\partial B_1)}^2 \\
			+\norm{u^{z_0,r}_i-Y^{z_0}_i}_{L^2(\partial B_1)}^2+\norm{u^{x_0,r}_i-u^{z_0,r}_i}_{L^2(\partial B_1)}^2\Big).
		\end{multline}
		The first two terms can be bounded in view of \Cref{prop:rate} as follows
		\begin{equation}\label{eq:osc_2}
			\sum_{i=1}^N\left(\norm{u^{x_0,r}_i-Y^{x_0}_i}_{L^2(\partial B_1)}^2+\norm{u^{z_0,r}_i-Y^{z_0}_i}_{L^2(\partial B_1)}^2\right)\leq C\,r^\alpha,
		\end{equation}
		for $r$ sufficiently small depending on $d$ and $\d_{\mathcal{H}}(K,\partial B_1)$ and $C>0$ depending on $d$ and $\sup_{x\in K}\mathcal{N}(u,x,r_0)$.
		On the other hand, if $L_K>0$ denotes the Lipschitz constant of $u$ in $K$, we have that
		\begin{equation}\label{eq:osc_3}
			\sum_{i=1}^N\norm{u^{x_0,r}_i-u^{z_0,r}_i}_{L^2(\partial B_1)}^2\leq \frac{L_K^2}{r^3}|x_0-z_0|^2.
		\end{equation}
		Therefore, we can choose $r$ in such a way that
		\begin{equation*}
			r^\alpha=\frac{|x_0-z_0|^2}{r^3},\quad\text{that is}\quad r=|x_0-z_0|^\frac{2}{\alpha+3}.
		\end{equation*}
		The claim follows by combining \eqref{eq:osc_3} and \eqref{eq:osc_2} with \eqref{eq:osc_1}.
	\end{proof}

\subsection{\texorpdfstring{Regularity of $\mathcal F_{\sfrac32}$}{Regularity of the lowest stratum of the singular set}}\label{sub:proof-regularity}
In this section we conclude the proof of \cref{t:main-main}. 

\begin{lemma}[Frequency gap from above]\label{lemma:gap_from_above}
		There exists $\delta_d>0$ such that, if $u\in \mathcal{M}_\gamma(\R^d;N)$ is a $\gamma$-homogeneous minimizer with $\gamma>3/2$, then there holds $\gamma\geq 3/2+\delta_d$.
	\end{lemma}
\begin{proof}
Suppose that this is not the case and that there is a sequence of $\gamma_n$ homogeneous solutions $u_n$ with $\gamma_n\to\sfrac32$. Then, up to a subsequence, $u_n$ converges to a $\sfrac32$-homogeneous solution $Y$. By \cref{lemma:eps_reg_outside}, for $n$ large enough we can apply the epiperimetric inequality from \cref{t:epi} to $u_n$. But this is a contradiction with the minimality of $u_n$. 
\end{proof}

	\begin{lemma}[No holes lemma]\label{l:no-holes}
		For any $\eps,\rho\in(0,\sfrac{1}{2})$ there exists $\delta>0$ (depending on $d$, $N$ and $\eps$) such that, if $u\in H^1(B_1;\Sigma_N)\cap C^{0,1}(B_1;\R^N)\cap\mathcal{M}(B_1;N)$ satisfies
		\begin{equation*}
			\norm{\d_{\Sigma_N}(u,Y)}_{L^\infty(B_1)}\leq \delta\quad\text{and}\quad \sup_{x\in B_{\frac{1}{2}}}\mathcal{N}\left(u,x,\frac{1}{2}\right)<\frac{3}{2}+\delta,
		\end{equation*}
		then for all $x''\in B_\rho'':=\{y\in\R^{d-2}\colon |y|<\rho\}$ there exists $x=(x'',x_{d-1},x_d)$ such that 
		$$\sqrt{x_{d-1}^2+x_d^2}<\eps\qquad\text{and}\qquad \gamma(u,x)=\sfrac{3}{2}.$$
	\end{lemma}    
	\begin{proof}
		We argue by contradiction and assume that there exist $\eps,\rho\in(0,\sfrac{1}{2})$, vanishing sequences $\{\delta_n\}\sub\R_+$ and $u_n\in  H^1(B_1;\Sigma_N)\cap C^{0,1}(B_1;\R^N)$ such that
		\begin{equation*}
			\norm{\d_{\Sigma_N}(u_n,Y)}_{L^\infty(B_1)}\leq \delta_n\quad\text{and}\quad \sup_{x\in B_{\frac{1}{2}}}\mathcal{N}\left(u,x,\frac{1}{2}\right)<\frac{3}{2}+\delta_n,
		\end{equation*}
		for all $n\in\N$ and there exist $x''_0\in B_{\rho}''$ such that there holds $\gamma(u_n,x)\neq \sfrac{3}{2}$ for all points $x\in \mathcal F(u_n)$ of the form $x=(x''_0,x_{d-1},x_d)\in B_{\rho}\cap C_\eps$. In view of \Cref{lemma:gap_from_above} and the fact that
		\begin{equation*}
			\gamma(u_n,x)\leq \mathcal{N}\left(u_n,x,\frac{1}{2}\right)<\frac{3}{2}+\delta_n,
		\end{equation*}
		for all $x\in B_{\rho}\cap C_\eps\cap\mathcal{F}(u_n)$, we derive that $\gamma(u_n,x)<\sfrac{3}{2}$ for $n$ sufficiently large for all $x\in B_{\rho}\cap C_\eps\cap\mathcal{F}(u_n)$. On the other hand, by the optimal frequency gap we also know that $\gamma(u_n,x)=1$  for all $x=(x''_0,x_{d-1},x_d)\in B_{\rho}\cap C_\eps\cap\mathcal{F}(u_n)$. As a consequence, for some $n$ large enough, we can find $t>0$ such that in the set 
		$$\mathcal U:=\Big\{(x'',x_{d-1},x_d)\ :\ |x''-x''_0|<t,\ \sqrt{x_{d-1}^2+x_d^2}<\frac12\Big\},$$
		the whole free boundary $\mathcal F(u_n)$ consists only of points of frequency $1$ and the number of connected components of $\mathcal U\setminus \mathcal F(u_n)$ is finite. Let now $\omega$ be the connected component containing $B_{1/2}\cap\{u_{n,1}>0\}\setminus \mathcal C_{\eps}$ and let $y\in B_{1/2}\cap\{u_{n,1}>0\}\setminus \mathcal C_{\eps}$ be fixed. By \cref{l:topological-lemma} (applied to $\mathcal M=\mathcal U$ and $\mathcal F=\mathcal F(u_n)$) every closed curve starting from $y$, contained in $\mathcal U$, and intersecting $\mathcal F(u_n)$ a finite number of times (and in a transversal way) should cross $\mathcal F(u_n)$ an even number of times. This is a contradiction since a closed curve circling around $\mathcal C_\eps$ crosses $\mathcal F(u_n)$ exactly three times, in view of \Cref{lemma:eps_reg_outside}.
	\end{proof}

\begin{proof}[\bf Proof of \cref{t:main-main}]
Let $x_0\in\mathcal F_{\sfrac32}{(u)}$ and let $u_0$ be the $\sfrac32$-homogeneous blow-up of $u$ at $x_0$, which is unique and non-zero by \cref{prop:rate}. Up to multiplying $u$ with a constant, we can suppose that $u_0=Y$. By \cref{lemma:gap_from_above} we know that in a neighborhood of $x_0$ there are only points of frequency $1$ or $\sfrac32$. Moreover, without loss of generality, we can take $x_0=0$. We will prove that there are $\rho>0$ and a $C^{1,\alpha}$ function $\eta:B_\rho''\to D_\rho$, where $D_\rho\subset\R^2$ is the disk of radius $\rho$ in $\R^2$, such that 
$$\mathcal F_{\sfrac3/2}(u)\cap \Big(B_\rho''\times D_\rho\Big)=\Big\{(x'',\eta(x''))\ :\ x''\in B_\rho''\Big\}.$$
We proceed in several steps. 

\noindent{\bf Step 1.} There are $\rho>0$, $C>0$ and a modulus of continuity $\bar\omega$ such that, for every $p\in B_\rho''\times D_\rho$ there is a $(d-2)\times 2$ real matrix $M_{p}$ such that 
\begin{multline}\label{e:matrix-M-differentiability}
B_r(p)\cap\mathcal F_{\sfrac32}(u)\subset \\
\Big\{(x'',x_{d-1},x_d)\colon x''\in B_r''(p''),\ |M_{p}[x''-p'']-(x_{d-1}-p_{d-1},x_d-p_d)|\le \bar\omega(r)\Big\},
\end{multline}
for every $r\in(0,\rho)$. Moreover, it holds
\begin{equation}\label{e:matrix-M-oscillation}
\|M_{p}-M_{q}\|\le C|p-q|^{\bar\alpha}\quad\text{for all}\quad p,q\in B_{\rho}''\times D_\rho,
\end{equation}
where $\bar\alpha>0$ is as in \Cref{lemma:osc}. Let $p\in B_\rho''\times D_\rho$. 
Thanks to \cref{prop:rate} we know that there is a unique $Y^{p}$ such that
\[
\left\|\d_{\Sigma_N}(u^{p,r},Y^p)\right\|_{L^\infty( B_1)}\le C_{\tu{rate}}'r^{\alpha'},
\]
for all $r<R$, where $R$ is uniform in a neighborhood of $0$. By \cref{lemma:eps_reg_outside}, there is a modulus of continuity $\omega(r)$ such that, for $r$ small enough, it holds 
$$\mathcal F_{\sfrac32}(u)\cap B_r(p)\subset \Big\{x\in B_r(p)\ :\ \dist(x,{\rm Sing}\,(Y^{p}))<\omega(r)\Big\}.$$
On the other hand, by \cref{lemma:osc} (up to renaming the constant $C_{\tu{osc}}$), we have that for any $p,q\in B_\rho''\times D_\rho$ (including $q=0$)
$$\|\d_{\Sigma_N}(Y^{p},Y^{q})\|_{L^\infty(B_1)}\le C_{\tu{osc}}|p-q|^{\bar\alpha}.$$
Since the space of triple junction blow-ups is finite dimensional, for the singular sets of $Y^{p}$ and $Y^{q}$ we have 
$$\d_{\mathcal H}\big({\rm Sing}\,(Y^{p})\cap B_1,{\rm Sing}\,(Y^{q})\cap B_1\big)\le C_{\tu{osc}}'|p-q|^{\bar\alpha}.$$
Thus, for each $p\in B_\rho''\times D_\rho$, the singular set ${\rm Sing}\,(Y^{p})$ is the graph of an affine function from $\R^{d-2}$ to $\R^2$. This provides the existence of the matrix $M_{p}$ for which \eqref{e:matrix-M-differentiability} and \eqref{e:matrix-M-oscillation} hold, where we can choose the modulus of continuity in \eqref{e:matrix-M-differentiability} to be $\bar\omega(r)=\omega(r)+Cr^{\bar\alpha}$.\medskip

\noindent{\bf Step 2.} For $\rho$ small enough $\mathcal F_{\sfrac32}(u)\cap B_\rho$ is a graph of a function $\eta:B_\rho''\to \R^2$. Indeed, by the no-holes lemma \cref{l:no-holes}, we know that for every $x''\in\R^{d-2}$ with $|x''|$ small enough, there exist $(x_{d-1},x_d)\in\R^2$ such that $x=(x'',x_{d-1},x_d)\in {\mathcal{F}_{\sfrac32}(u)}$. On the other hand, given $x''$, the point $(x_{d-1},x_d)$ is unique, thanks to Step 1.\medskip

\noindent{\bf Step 3.} Regularity of $\eta$. First, we notice that thanks to \eqref{e:matrix-M-differentiability}, $\eta$ is differentiable and $M_{p}$ is precisely its Jacobian matrix at $p$. Thanks to \eqref{e:matrix-M-oscillation}, we get that $\eta\in C^{1,\alpha}$. 
\medskip

Finally, the regularity of $\mathcal F(u)$ around $0$ is a consequence of the regularity of $\eta$, the uniqueness of the blow-up, the control of the oscillation in \Cref{lemma:osc}, and the epsilon-regularity for $\mathcal F_1{(u)}$ (\cite[Lemma 9.8]{OV}); the argument is standard and for the details we refer for instance to \cite[Section 9]{OV}.
\end{proof}

\section{\texorpdfstring{Harmonic maps with values in $\R$-trees. Proof of \cref{t:harmonic-maps-graphs}}{Harmonic maps with values in R-trees}}\label{s:harmonic-maps-R-trees}

In \cref{sub:R-trees-definition,sub:R-trees-energy-minimizing,sub:R-trees-order-function,sub:R-trees-regular-singular} we introduce the main definitions needed in order to state and prove \Cref{t:harmonic-maps-graphs}. 

\subsection{\texorpdfstring{$\R$-trees}{R-trees}}\label{sub:R-trees-definition} We recall that an \emph{$\R$-tree} is a metric space $(T,\d_T)$ satisfying the following conditions: 
\begin{itemize}
\item for any $p,q\in T$ there is a unique injective arc $\sigma:[0,1]\to T$ joining $p$ and $q$;
\item moreover, the arc $\sigma$ is rectifiable and its length is precisely the distance $\d_T(p,q)$. 
\end{itemize}
Moreover, given a point $p\in T$ we say that: 
\begin{itemize}
    \item $p$ is a \emph{boundary point}, if $T\setminus\{p\}$ is connected;
    \item $p$ is an \emph{edge point}, if $T\setminus\{p\}$ has two connected components; 
    \item $p$ is a \emph{vertex}, if $T\setminus\{p\}$ has three or more connected components. 
\end{itemize}

A {\it locally finite tree} is a tree in which all the vertices are isolated and for every vertex $p$, $T\setminus\{p\}$ has finitely many connected components. An \emph{$N$-pod} is a tree with only one vertex $\{p\}$ and finitely many edges $E_1,\dots,E_N$ attached to it; for each $j=1,\dots,N$, there is a curve $\sigma:[0,\ell_j)\to T$, with $\ell_j\in(0,+\infty]$, such that:
\begin{itemize}
\item $\sigma$ is injective;
\item $\sigma$ is isometric: $|s-t|=\d_T\big(\sigma(s),\sigma(t))$ for all $s,t\in[0,\ell_j)$; \item $\sigma(0)=p$ and $\sigma((0,\ell_j))=E_j$. 
\end{itemize}
In particular, it is immediate to check that any $N$-pod $T_N$ is isometric to an open subset of $\Sigma_N$ (defined in \eqref{e:definition-of-Sigma-N}) containing the origin. Finally, we say that $T'\subset T$ if $T'$ is an $\R$-tree.
\medskip

\subsection{Energy-minimizing maps}\label{sub:R-trees-energy-minimizing} Given a Lipschitz domain $\Omega\subset\R^d$, a general $\R$-tree $(T,\d_T)$ and a map $f:\Omega\to T$, Korevaar and Schoen showed in \cite{KS} that the gradient of $f$ at $x\in\Omega$ can be defined as the density (at $x$) of the measure $|\nabla f(x)|^2\dx = \d e$ obtained as weak limit (as $\eps\to0$) of a sequence of absolutely continuous measures 
$e_\eps(x)\dx$ with densities
$$e_\eps(x):=\frac{1}{\omega_d\eps^{d+1}}\int_{\partial B_\eps(x)}\d_T^2(f(x),f(y))\,
\d\mathcal{H}^{d-1}(y).$$
One can then accordingly define the Sobolev space $W^{1,2}(\Omega;T)$, and prove the existence of a well-defined trace operator $W^{1,2}(\Omega;T)\hookrightarrow L^2(\partial\Omega;T)$, see \cite[Theorem 1.12.2]{KS}.

\medskip

\begin{definition}
We say that a map $u\in W^{1,2}(\Omega;T)$ is energy-minimizing if 
$$\int_{\Omega}|\nabla u|^2\dx\le \int_{\Omega}|\nabla g|^2\dx,$$
for all $g\in W^{1,2}(\Omega;T)$ with the same trace as $u$ on $\partial\Omega$. 
\end{definition}
\begin{remark}
It was proved in \cite{KS} (in the general case of non-positively curved targets) that any energy-minimizing map $u\in W^{1,2}(\Omega;T)$ is locally Lipschitz continuous in $\Omega$.
\end{remark}

\subsection{Order function}\label{sub:R-trees-order-function} In \cite{GS}, Gromov and Schoen introduced the \emph{order function} (analogous to the Almgren's frequency function) for any energy-minimizing $u\in W^{1,2}(\Omega;T)$\footnote{To be precise, in \cite{GS} the author considered the case in which the target is an euclidean building embedded in $\R^N$, but in view of the work \cite{KS} one can directly extend the definition to $\R$-trees, see \cite{Sun}.
}
\begin{equation*}
\mathrm{Ord}^u(x_0,r,p)=\frac{\displaystyle r\int_{B_r(x_0)}|\nabla u|^2\dx}{\displaystyle \int_{\partial B_r(x_0)}\d_T^2(u(x),p)\,\d\mathcal H^{d-1}(x)},
\end{equation*}
for any $p\in T$ and $x_0\in \Omega$. In \cite{Sun} it was shown that, for energy-minimizing maps $u$, the function $r\mapsto \mathrm{Ord}^u(x_0,r,p)$ is monotone increasing in $r$ and that, taking $p=u(x_0)$, the limit 
$$\mathrm{Ord}^u(x_0):=\lim_{r\to0^+}\mathrm{Ord}^u(x_0,r,u(x_0))$$
exists and is finite at every $x_0\in\Omega$. 

\subsection{Regular and singular points}\label{sub:R-trees-regular-singular} We now recall some known results which lead to the proof of \Cref{t:harmonic-maps-graphs}. Thanks to \cite{GS,Sun}, we know that every energy-minimizing map $u:\Omega\to T$ has the following properties:
\begin{itemize}
\item {\it Frequency gap \rm(\cite{GS}).} There exists $\delta_d>0$ such that, for every $x_0\in\Omega$, we have:  
\begin{center}
either $\mathrm{Ord}^u(x_0)= 1$ or $\mathrm{Ord}^u(x_0)\ge 1+\delta_d$.\smallskip
\end{center}
\item {\it Epsilon-regularity \rm(\cite{GS}).} The set 
\[
\mathrm{Reg}(u):=\{x\in \Omega\ :\ \mathrm{Ord}^u(x)=1\},
\]
is an open subset of $\Omega$. For every $x_0\in \mathrm{Reg}(u)$, there is $r>0$ such that $u(B_r(x_0))$ lies in an embedded arc of $T$, which is isometric to an interval $(a,b)\subset\R$. Moreover, $u:B_r(x_0)\to(a,b)$ is harmonic in $B_r(x_0)$ and $\nabla u(x_0)\neq 0$.\smallskip
\item {\it Hausdorff dimension of the singular set \rm(\cite{GS}).} The set 
\[
\mathrm{Sing}(u):=\Omega\setminus \mathrm{Reg}(u)=\{x\in\Omega\ : \ \mathrm{Ord}^u(x)\geq 1+\delta_d\},
\]
is a closed subset of $\Omega$ of Hausdorff dimension at most $d-2$.\smallskip 
\item {\it Clean-up at singular points \rm(\cite{Sun}).} For every $x_0\in \mathrm{Sing}(u)$, there is $r>0$ such that $u(B_r(x_0))$ lies in an (embedded) $N$-pod $T_N\subset T$ (see \cite[Theorem 1.4]{Sun}). 
\end{itemize}
\subsection{\texorpdfstring{Proof of \cref{t:harmonic-maps-graphs}}{Proof of the main theorem}}\label{sub:R-trees-proof-of-the-theorem} 
Let $u:B_1\to T$ be an energy-minimizing map. Let $x_0\in \textrm{Sing}(u)$. Thanks to the clean-up result from \cite{Sun}, there are a ball $B_r(x_0)$ and an $N$-pod  $T_N\subset T$ with vertex $p_0\in T_N$, such that $u(B_r(x_0))\subset T_N$ and $u(x_0)=p_0$, and such that $u:B_r(x_0)\to T_N$ is an energy-minimizing map. There is a map $\Phi:T_N\to \Sigma_N$, which is an isometry between the $N$-pod $(T_N,\d_T)$ and the subset $\Phi(T_N)$ of $\Sigma_N$ equipped with the metric $\d_{\Sigma_N}$. Moreover, we can suppose that $\Phi(p_0)=0$. We consider the map $\widetilde u:B_r(x_0)\to \Phi(T_N)$, $\widetilde u=u\circ \Phi^{-1}$. Then, $\widetilde u=(\widetilde u_1,\dots,\widetilde u_N):B_r(x_0)\to T_N$ and  the energy of $u$ can be written in the form  
$$\int_{B_r(x_0)}|\nabla u|^2=\sum_{i=1}^N\int_{B_r(x_0)}|\nabla \widetilde u_i|^2\dx.$$
Since any map $\widetilde g:B_r(x_0)\to\Phi(T_N)$ can be written in the form $\widetilde g=g\circ \Phi^{-1}$ the energy minimizing property of $u:B_r(x_0)\to T_N$ translates into
$$\sum_{i=1}^N\int_{B_r(x_0)}|\nabla \widetilde u_i|^2\dx\le \sum_{i=1}^N\int_{B_r(x_0)}|\nabla \widetilde g_i|^2\dx,$$
for all $\widetilde g=(\widetilde g_1,\dots,\widetilde g_N):B_r(x_0)\to \Phi(T_N)\subset \Sigma_N$ with $\widetilde g_i-\widetilde u_i\in H^1_0(B_r(x_0))$ for $i=1,\dots,N$. Moreover, we have that, for every $x\in B_r(x_0)\cap u^{-1}(\{p_0\})$, the Order function of $u$ at $x$ coincides with the Almgren's frequency function of $\widetilde u$ at $x$:
$$\text{Ord}^{u}(x)=\gamma(\widetilde u,x).$$
Applying \cref{t:main-main} to the map $\widetilde u:B_r(x_0)\to\Sigma_N$ we obtain the decomposition 
\begin{align*}
B_r(x_0)=&\Big\{x\in B_r(x_0)\ :\ \widetilde u(x)=0\ ,\ \gamma(\widetilde u,x)=1\Big\}\\
&\quad\cup \Big\{x\in B_r(x_0)\ :\ \widetilde u(x)=0\ ,\ \gamma(\widetilde u,x)=\frac32\Big\}\\
&\qquad\cup\Big\{x\in B_r(x_0)\ :\ \widetilde u(x)=0\ ,\ \gamma(\widetilde u,x)>\frac32+\delta_d\Big\},
\end{align*}
as well as the regularity of 
$$\mathcal F_{\sfrac32}\cap B_r(x_0):=\Big\{x\in B_r(x_0)\ :\ \widetilde u(x)=0\ ,\ \gamma(\widetilde u,x)=\frac32\Big\}.$$
On the other hand, since $p_0$ is the only vertex of $T_N$, we have that, for every $x\in B_r(x_0)\setminus u^{-1}(\{p_0\})$, there is a neighborhood $B_\rho(x)$, such that $u(B_\rho(x))$ is contained in an edge of $T_N$, which is isometric to an open segment of $\R$. Thus, there is a map $\Psi:B_r(x)\cap T\to \R$, which is an isometry and for which the map $u\circ\Psi^{-1}:B_r(x)\to\R$ is harmonic. Thus, for all $x\in B_r(x_0)\setminus u^{-1}(\{p_0\})$, we get
$$\text{Ord}^{u}(x)=1\quad\text{or}\quad \text{Ord}^{u}(x)\ge 2.$$
This implies that the set $\mathcal F_{\sfrac32} \cap B_r(x_0)$ defined above can also be written as:
$$\mathcal F_{\sfrac32}\cap B_r(x_0)=\Big\{x\in B_r(x_0)\ :\ \text{Ord}^{u}(x)=\frac32\Big\},$$
and, as a consequence, we obtain the following decomposition of the entire ball $B_r(x_0)$:
\begin{align*}
B_r(x_0)&=\Big\{x\in B_r(x_0)\ :\ \text{Ord}^{u}(x)=1\Big\}\cup \mathcal F_{\sfrac32}\cup\Big\{x\in B_r(x_0)\ :\ \text{Ord}^{u}(x)>\frac32+\delta_d\Big\},
\end{align*}
which concludes the proof of \cref{t:harmonic-maps-graphs}. \qed

	\subsection*{Acknowledgements}

	The authors are supported by the European Research Council (ERC), through the European Union’s Horizon 2020 project ERC VAREG - Variational approach to the regularity of the free boundaries (grant agreement No. 853404). The authors acknowledge the MIUR Excellence Department Project awarded to the Department of Mathematics, University of Pisa, CUP I57G22000700001. R.O. acknowledges support from the 2024 INdAM-GNAMPA project no. \texttt{CUP\_E53C23001670001} and from the 2026 INdAM-GNAMPA project no. \texttt{CUP\_E53C25002010001}. B.V. acknowledges support from the projects PRA\_2022\_14 GeoDom (PRA 2022 - Università di Pisa) and MUR-PRIN “NO3” (n.2022R537CS). 

    The authors would like to thank the anonymous referees for the careful reading of the paper, which helped to improve the exposition and to fix an issue with the proof of \Cref{t:epi} in a previous version of the paper. The authors are also grateful to Gianmaria Verzini for the useful discussion and for pointing out the reference \cite{wang-zhang}.

    \subsection*{Corresponding author} Bozhidar Velichkov (e-mail: bozhidar.velichkov@unipi.it)
\subsection*{Conflict of interests and data availability statements} The authors have no competing interests to declare. Data sharing is not applicable to this article as no datasets were generated or analyzed during the current study.

\subsection*{Further statements} This preprint has not undergone peer review (when applicable) or any post-submission improvements or corrections. The Version of Record of this article is published in \emph{Arch. Rational Mech. Anal.}, and is available online at \url{https://doi.org/10.1007/s00205-026-02221-4}
	\bibliographystyle{aomalpha} 
	\bibliography{biblio.bib}

\providecommand{\etalchar}[1]{$^{#1}$}
\providecommand{\bysame}{\leavevmode\hbox to3em{\hrulefill}\thinspace}
\providecommand{\noopsort}[1]{}
\providecommand{\mr}[1]{\href{http://www.ams.org/mathscinet-getitem?mr=#1}{MR~#1}}
\providecommand{\zbl}[1]{\href{http://www.zentralblatt-math.org/zmath/en/search/?q=an:#1}{Zbl~#1}}
\providecommand{\jfm}[1]{\href{http://www.emis.de/cgi-bin/JFM-item?#1}{JFM~#1}}
\providecommand{\arxiv}[1]{\href{http://www.arxiv.org/abs/#1}{arXiv~#1}}
\providecommand{\doi}[1]{\url{https://doi.org/#1}}
\providecommand{\MR}{\relax\ifhmode\unskip\space\fi MR }
\providecommand{\MRhref}[2]{%
  \href{http://www.ams.org/mathscinet-getitem?mr=#1}{#2}
}
\providecommand{\href}[2]{#2}
\begin{thebibliography}{AMdSST24}

\bibitem[Alp20]{Alper}
\bgroup\scshape{}O.~Alper\egroup{}, On the singular set of free interface in an optimal partition problem,  \emph{Commun. Pure Appl. Math.} \textbf{73} no.~4 (2020), 855--915 (English). \doi{10.1002/cpa.21874}.

\bibitem[AMdSST24]{tavares_2}
\bgroup\scshape{}P.~D.~S. Andrade\egroup{}, \bgroup\scshape{}E.~Moreira~dos Santos\egroup{}, \bgroup\scshape{}M.~S. Santos\egroup{}, and \bgroup\scshape{}H.~Tavares\egroup{}, Spectral partition problems with volume and inclusion constraints,  \emph{SIAM J. Math. Anal.} \textbf{56} no.~6 (2024), 7136--7169 (English). \doi{10.1137/23M161553X}.

\bibitem[Bis92]{Bishop}
\bgroup\scshape{}C.~J. Bishop\egroup{}, Some questions concerning harmonic measure,  in \emph{Partial differential equations with minimal smoothness and applications ({C}hicago, {IL}, 1990)}, \emph{IMA Vol. Math. Appl.} \textbf{42}, Springer, New York, 1992, pp.~89--97. \mr{1155854}.  \doi{10.1007/978-1-4612-2898-1\_7}.

\bibitem[BBH98]{BBH1998}
\bgroup\scshape{}D.~Bucur\egroup{}, \bgroup\scshape{}G.~Buttazzo\egroup{}, and \bgroup\scshape{}A.~Henrot\egroup{}, Existence results for some optimal partition problems,  \emph{Adv. Math. Sci. Appl.} \textbf{8} no.~2 (1998), 571--579 (English).

\bibitem[CL08]{CL2008}
\bgroup\scshape{}L.~A. Caffarelli\egroup{} and \bgroup\scshape{}F.-H. Lin\egroup{}, Singularly perturbed elliptic systems and multi-valued harmonic functions with free boundaries,  \emph{J. Am. Math. Soc.} \textbf{21} no.~3 (2008), 847--862 (English). \doi{10.1090/S0894-0347-08-00593-6}.

\bibitem[CL10]{CL2010}
\bgroup\scshape{}L.~A. Caffarelli\egroup{} and \bgroup\scshape{}F.~H. Lin\egroup{}, Analysis on the junctions of domain walls,  \emph{Discrete Contin. Dyn. Syst.} \textbf{28} no.~3 (2010), 915--929. \mr{2644773}.  \doi{10.3934/dcds.2010.28.915}.

\bibitem[CL07]{CL2007}
\bgroup\scshape{}L.~A. Cafferelli\egroup{} and \bgroup\scshape{}F.~H. Lin\egroup{}, An optimal partition problem for eigenvalues,  \emph{J. Sci. Comput.} \textbf{31} no.~1-2 (2007), 5--18. \mr{2304268}.  \doi{10.1007/s10915-006-9114-8}.

\bibitem[CTV03]{CTV2003}
\bgroup\scshape{}M.~Conti\egroup{}, \bgroup\scshape{}S.~Terracini\egroup{}, and \bgroup\scshape{}G.~Verzini\egroup{}, An optimal partition problem related to nonlinear eigenvalues,  \emph{J. Funct. Anal.} \textbf{198} no.~1 (2003), 160--196. \mr{1962357}.  \doi{10.1016/S0022-1236(02)00105-2}.

\bibitem[CTV05a]{CTVFucick2005}
\bgroup\scshape{}M.~Conti\egroup{}, \bgroup\scshape{}S.~Terracini\egroup{}, and \bgroup\scshape{}G.~Verzini\egroup{}, On a class of optimal partition problems related to the {F}ucik spectrum and to the monotonicity formulae,  \emph{Calc. Var. Partial Differential Equations} \textbf{22} no.~1 (2005), 45--72. \mr{2105968}.  \doi{10.1007/s00526-004-0266-9}.

\bibitem[CTV05b]{CTV2005indiana}
\bgroup\scshape{}M.~Conti\egroup{}, \bgroup\scshape{}S.~Terracini\egroup{}, and \bgroup\scshape{}G.~Verzini\egroup{}, A variational problem for the spatial segregation of reaction-diffusion systems,  \emph{Indiana Univ. Math. J.} \textbf{54} no.~3 (2005), 779--815. \mr{2151234}.  \doi{10.1512/iumj.2005.54.2506}.

\bibitem[DHM{\etalchar{+}}22]{modp-3}
\bgroup\scshape{}C.~{De Lellis}\egroup{}, \bgroup\scshape{}J.~Hirsch\egroup{}, \bgroup\scshape{}A.~Marchese\egroup{}, \bgroup\scshape{}L.~Spolaor\egroup{}, and \bgroup\scshape{}S.~Stuvard\egroup{}, Fine structure of the singular set of area minimizing hypersurfaces modulo $p$,  (2022). \arxiv{arXiv:2201.10204}.  Available at \url{https://arxiv.org/abs/2201.10204}.

\bibitem[DLHM{\etalchar{+}}26]{modp-2}
\bgroup\scshape{}C.~De~Lellis\egroup{}, \bgroup\scshape{}J.~Hirsch\egroup{}, \bgroup\scshape{}A.~Marchese\egroup{}, \bgroup\scshape{}L.~Spolaor\egroup{}, and \bgroup\scshape{}S.~Stuvard\egroup{}, Area minimizing hypersurfaces modulo {{\(p\)}}: a geometric free-boundary problem,  \emph{J. Funct. Anal.} \textbf{290} no.~12 (2026), 111 (English), Id/No 111442. \doi{10.1016/j.jfa.2026.111442}.

\bibitem[DHMS20]{modp-1}
\bgroup\scshape{}C.~{De Lellis}\egroup{}, \bgroup\scshape{}J.~Hirsch\egroup{}, \bgroup\scshape{}A.~Marchese\egroup{}, and \bgroup\scshape{}S.~Stuvard\egroup{}, Regularity of area minimizing currents mod {{\(p\)}},  \emph{Geom. Funct. Anal.} \textbf{30} no.~5 (2020), 1224--1336 (English). \doi{10.1007/s00039-020-00546-0}.

\bibitem[DSS24]{DeSilvaSavin2024}
\bgroup\scshape{}D.~De~Silva\egroup{} and \bgroup\scshape{}O.~Savin\egroup{}, An energy model for harmonic functions with junctions,  \emph{Adv. Math.} \textbf{447} (2024), 55 (English), Id/No 109682. \doi{10.1016/j.aim.2024.109682}.

\bibitem[Dee22]{Dees}
\bgroup\scshape{}B.~K. Dees\egroup{}, Rectifiability of the singular set of harmonic maps into buildings,  \emph{J. Geom. Anal.} \textbf{32} no.~7 (2022), Paper No. 205, 57. \mr{4425367}.  \doi{10.1007/s12220-022-00943-x}.

\bibitem[FS16]{FSthin}
\bgroup\scshape{}M.~Focardi\egroup{} and \bgroup\scshape{}E.~Spadaro\egroup{}, An epiperimetric inequality for the thin obstacle problem,  \emph{Adv. Differ. Equ.} \textbf{21} no.~1-2 (2016), 153--200 (English).

\bibitem[FH76]{FH}
\bgroup\scshape{}S.~Friedland\egroup{} and \bgroup\scshape{}W.~K. Hayman\egroup{}, Eigenvalue inequalities for the {D}irichlet problem on spheres and the growth of subharmonic functions,  \emph{Comment. Math. Helv.} \textbf{51} no.~2 (1976), 133--161. \mr{412442}.  \doi{10.1007/BF02568147}.

\bibitem[GPS16]{GPGthin}
\bgroup\scshape{}N.~Garofalo\egroup{}, \bgroup\scshape{}A.~Petrosyan\egroup{}, and \bgroup\scshape{}M.~{Smit Vega Garcia}\egroup{}, An epiperimetric inequality approach to the regularity of the free boundary in the {Signorini} problem with variable coefficients,  \emph{J. Math. Pures Appl. (9)} \textbf{105} no.~6 (2016), 745--787 (English). \doi{10.1016/j.matpur.2015.11.013}.

\bibitem[GS92]{GS}
\bgroup\scshape{}M.~Gromov\egroup{} and \bgroup\scshape{}R.~Schoen\egroup{}, Harmonic maps into singular spaces and {$p$}-adic superrigidity for lattices in groups of rank one,  \emph{Inst. Hautes \'Etudes Sci. Publ. Math.} no.~76 (1992), 165--246. \mr{1215595}.  Available at \url{http://www.numdam.org/item?id=PMIHES_1992__76__165_0}.

\bibitem[Hel10]{Helffer-survey}
\bgroup\scshape{}B.~Helffer\egroup{}, On spectral minimal partitions: a survey,  \emph{Milan J. Math.} \textbf{78} no.~2 (2010), 575--590. \mr{2781853}.  \doi{10.1007/s00032-010-0129-0}.

\bibitem[HHT09]{HHOT09}
\bgroup\scshape{}B.~Helffer\egroup{}, \bgroup\scshape{}T.~{Hoffmann-Ostenhof}\egroup{}, and \bgroup\scshape{}S.~Terracini\egroup{}, Nodal domains and spectral minimal partitions,  \emph{Ann. Inst. H. Poincar\'e{} C Anal. Non Lin\'eaire} \textbf{26} no.~1 (2009), 101--138. \mr{2483815}.  \doi{10.1016/j.anihpc.2007.07.004}.

\bibitem[HH10]{HH10}
\bgroup\scshape{}B.~Helffer\egroup{} and \bgroup\scshape{}T.~{Hoffmann-Ostenhof}\egroup{}, Remarks on two notions of spectral minimal partitions,  \emph{Adv. Math. Sci. Appl.} \textbf{20} no.~1 (2010), 249--263. \mr{2760728}.

\bibitem[HHT10]{HHOT10}
\bgroup\scshape{}B.~Helffer\egroup{}, \bgroup\scshape{}T.~{Hoffmann-Ostenhof}\egroup{}, and \bgroup\scshape{}S.~Terracini\egroup{}, On spectral minimal partitions: the case of the sphere,  in \emph{Around the research of {V}ladimir {M}az'ya. {III}}, \emph{Int. Math. Ser. (N. Y.)} \textbf{13}, Springer, New York, 2010, pp.~153--178. \mr{2664708}.  \doi{10.1007/978-1-4419-1345-6\_6}.

\bibitem[HKT25]{tavares_1}
\bgroup\scshape{}M.~Hofmann\egroup{}, \bgroup\scshape{}J.~B. Kennedy\egroup{}, and \bgroup\scshape{}H.~Tavares\egroup{}, Spectral minimal partitions of unbounded domains,  (2025). \arxiv{arXiv:2510.00811}.  Available at \url{https://arxiv.org/abs/2510.00811}.

\bibitem[KS93]{KS}
\bgroup\scshape{}N.~J. Korevaar\egroup{} and \bgroup\scshape{}R.~M. Schoen\egroup{}, Sobolev spaces and harmonic maps for metric space targets,  \emph{Comm. Anal. Geom.} \textbf{1} no.~3-4 (1993), 561--659. \mr{1266480}.  \doi{10.4310/CAG.1993.v1.n4.a4}.

\bibitem[LMNS24]{spadaro-generic}
\bgroup\scshape{}F.~Lanzara\egroup{}, \bgroup\scshape{}E.~Montefusco\egroup{}, \bgroup\scshape{}V.~Nesi\egroup{}, and \bgroup\scshape{}E.~Spadaro\egroup{}, Generic configurations in 2d strongly competing systems,  \emph{Preprint} (2024). Available at \url{https://arxiv.org/abs/2411.19867}.

\bibitem[MST24]{tavares_3}
\bgroup\scshape{}D.~Mazzoleni\egroup{}, \bgroup\scshape{}M.~S. Santos\egroup{}, and \bgroup\scshape{}H.~Tavares\egroup{}, Free boundary regularity for a spectral optimal partition problem with volume and inclusion constraints, Preprint, {arXiv}:2409.14916 [math.{AP}], 2024. \arxiv{arXiv:2409.14916}.  Available at \url{https://arxiv.org/abs/2409.14916}.

\bibitem[NRS24]{NRS24}
\bgroup\scshape{}M.~Novack\egroup{}, \bgroup\scshape{}D.~Restrepo\egroup{}, and \bgroup\scshape{}A.~Skorobogatova\egroup{}, Free boundary regularity for semilinear variational problems with a topological constraint,  \emph{Accepted in Arch. Ration. Mech. Anal.} (2024). Available at \url{https://arxiv.org/abs/2412.01813}.

\bibitem[OV24]{OV}
\bgroup\scshape{}R.~Ognibene\egroup{} and \bgroup\scshape{}B.~Velichkov\egroup{}, Boundary regularity of the free interface in spectral optimal partition problems,  \emph{Preprint} (2024). \arxiv{arXiv:2404.05698}.  Available at \url{https://arxiv.org/abs/2404.05698}.

\bibitem[SY19]{SavinYu2019}
\bgroup\scshape{}O.~Savin\egroup{} and \bgroup\scshape{}H.~Yu\egroup{}, On the multiple membranes problem,  \emph{J. Funct. Anal.} \textbf{277} no.~6 (2019), 1581--1602 (English). \doi{10.1016/j.jfa.2019.06.003}.

\bibitem[SY21]{SavinYu2021}
\bgroup\scshape{}O.~Savin\egroup{} and \bgroup\scshape{}H.~Yu\egroup{}, Free boundary regularity in the triple membrane problem,  \emph{Ars Inven. Anal.} \textbf{2021} (2021), 49 (English), Id/No 3. \doi{10.15781/ys6e-4d80}.

\bibitem[SY23]{SavinYu2023}
\bgroup\scshape{}O.~Savin\egroup{} and \bgroup\scshape{}H.~Yu\egroup{}, Free boundary regularity in the multiple membrane problem in the plane,  \emph{J. Reine Angew. Math.} \textbf{799} (2023), 109--154 (English). \doi{10.1515/crelle-2023-0014}.

\bibitem[Sim93]{SimonJDG}
\bgroup\scshape{}L.~Simon\egroup{}, Cylindrical tangent cones and the singular set of minimal submanifolds,  \emph{J. Differ. Geom.} \textbf{38} no.~3 (1993), 585--652 (English). \doi{10.4310/jdg/1214454484}.

\bibitem[ST15]{ST-opt-gap}
\bgroup\scshape{}N.~Soave\egroup{} and \bgroup\scshape{}S.~Terracini\egroup{}, Liouville theorems and 1-dimensional symmetry for solutions of an elliptic system modelling phase separation,  \emph{Adv. Math.} \textbf{279} (2015), 29--66. \mr{3345178}.  \doi{10.1016/j.aim.2015.03.015}.

\bibitem[ST24a]{ST_partially_1}
\bgroup\scshape{}N.~Soave\egroup{} and \bgroup\scshape{}S.~Terracini\egroup{}, On partially segregated harmonic maps: optimal regularity and structure of the free boundary,  (2024). \arxiv{arXiv:2410.23976}.  Available at \url{https://arxiv.org/abs/2410.23976}.

\bibitem[ST24b]{ST24a}
\bgroup\scshape{}N.~Soave\egroup{} and \bgroup\scshape{}S.~Terracini\egroup{}, On partially segregated harmonic maps: optimal regularity and structure of the free boundary, 2024. Available at \url{https://arxiv.org/abs/2410.23976}.

\bibitem[ST24c]{ST_partially_2}
\bgroup\scshape{}N.~Soave\egroup{} and \bgroup\scshape{}S.~Terracini\egroup{}, On some singularly perturbed elliptic systems modeling partial segregation: uniform {H{\"o}lder} estimates and basic properties of the limits,  (2024). \arxiv{arXiv:2409.11976}.  Available at \url{https://arxiv.org/abs/2409.11976}.

\bibitem[ST24d]{ST24}
\bgroup\scshape{}N.~Soave\egroup{} and \bgroup\scshape{}S.~Terracini\egroup{}, On some singularly perturbed elliptic systems modeling partial segregation: uniform h\"older estimates and basic properties of the limits,  \emph{Preprint} (2024). Available at \url{https://arxiv.org/abs/2409.11976}.

\bibitem[Sun03]{Sun}
\bgroup\scshape{}X.~Sun\egroup{}, Regularity of harmonic maps to trees,  \emph{Am. J. Math.} \textbf{125} no.~4 (2003), 737--771 (English). \doi{10.1353/ajm.2003.0029}.

\bibitem[TT12]{TerraciniTavares2012}
\bgroup\scshape{}H.~Tavares\egroup{} and \bgroup\scshape{}S.~Terracini\egroup{}, Regularity of the nodal set of segregated critical configurations under a weak reflection law,  \emph{Calc. Var. Partial Differential Equations} \textbf{45} no.~3-4 (2012), 273--317. \mr{2984134}.  \doi{10.1007/s00526-011-0458-z}.

\bibitem[Tay73]{TaylorInventiones}
\bgroup\scshape{}J.~E. Taylor\egroup{}, Regularity of the singular sets of two-dimensional area-minimizing flat chains modulo 3 in {R}{{\(^3\)}},  \emph{Invent. Math.} \textbf{22} (1973), 119--159 (English). \doi{10.1007/BF01392299}.

\bibitem[Tay76]{TaylorAnnals}
\bgroup\scshape{}J.~E. Taylor\egroup{}, The structure of singularities in soap-bubble-like and soap-film-like minimal surfaces,  \emph{Ann. Math. (2)} \textbf{103} (1976), 489--539 (English). \doi{10.2307/1970949}.

\bibitem[WZ10]{wang-zhang}
\bgroup\scshape{}K.~Wang\egroup{} and \bgroup\scshape{}Z.~Zhang\egroup{}, Some new results in competing systems with many species,  \emph{Ann. Inst. Henri Poincar{\'e}, Anal. Non Lin{\'e}aire} \textbf{27} no.~2 (2010), 739--761 (English). \doi{10.1016/j.anihpc.2009.11.004}.

\bibitem[Wei99]{Weiss}
\bgroup\scshape{}G.~S. Weiss\egroup{}, A homogeneity improvement approach to the obstacle problem,  \emph{Invent. Math.} \textbf{138} no.~1 (1999), 23--50 (English). \doi{10.1007/s002220050340}.

\end{thebibliography}

\end{document}